\documentclass[reqno, 10pt, a4paper]{amsart}
\makeatletter
\renewcommand\section{\@startsection{section}{1}%
  \z@{.7\linespacing\@plus\linespacing}{.5\linespacing}%
  {\normalfont\bfseries\centering}}
\makeatother
 
\makeatletter
\renewcommand\paragraph{\@startsection{paragraph}{4}%
  \z@\z@{-\fontdimen2\font}%
  {\normalfont\bfseries}}
\makeatother
 
\usepackage[T1]{fontenc}
\usepackage{ucs}
\usepackage[utf8x]{inputenc}
 
\usepackage[dvipsnames]{xcolor}
\usepackage[reqno,intlimits]{mathtools}
\usepackage{amssymb, MnSymbol, accents}
\usepackage{amsthm}
\usepackage{empheq}
\usepackage{mathrsfs, dsfont}
\usepackage{esint}
\usepackage{enumitem}
\newlist{conditions}{enumerate}{1}
\setlist[conditions]{label=(\alph*),ref=(\alph*)}
\makeatletter
\newcommand{\mylabel}[2]{#2\def\@currentlabel{#2}\label{#1}}
\makeatother
\usepackage[final]{graphicx}
\usepackage{float, subcaption, caption, wrapfig}
\usepackage{cutwin}
\usepackage{everysel, stackrel}
\usepackage[draft,textsize=footnotesize]{todonotes}

\usepackage[text={32pc,610pt},centering]{geometry}
 
\usepackage[bookmarks,colorlinks,citecolor=blue]{hyperref}
 
\newtheorem{thm}{Theorem}[section]
\newtheorem{cor}[thm]{Corollary}
\newtheorem{lem}[thm]{Lemma}
\newtheorem{prop}[thm]{Proposition}
\newtheorem{ass}[thm]{Assumption}
\newtheorem{remark}[thm]{Remark}

\newtheoremstyle{definition}%
{}{}%
{\itshape}{}%
{\bfseries}{.}%
{\newline}%
{\thmname{#1}\thmnumber{ #2}:\thmnote{ #3}}
 
\theoremstyle{definition}
\newtheorem{defn}[thm]{Definition}

\def\smallskip{\vskip\smallskipamount}
\def\medskip{\vskip\medskipamount}
\def\bigskip{\vskip\bigskipamount}
 
\newcounter{wraprem}[section]

\def\claim{\par\medskip\noindent\refstepcounter{wraprem}\hbox{\bf Remark \arabic{section}.\arabic{wraprem}}
\it\ 
}
\def\endclaim{
\par\medskip}
\newenvironment{wraprem}{\claim}{\endclaim}

\expandafter\let\expandafter\oldproof\csname\string\proof\endcsname
\let\oldendproof\endproof
\renewenvironment{proof}[1][\proofname]{%
  \oldproof[{\bf #1}]%
}{\oldendproof}

\let\hom\relax
\newcommand{\hom}{\rm hom}

\renewcommand{\Pr}{\mathbb{P}}
\newcommand{\Pas}{$\Pr$-a.s.\ }
\newcommand{\Pae}{$\Pr$-a.e.\ }
\newcommand{\Prw}{\mathrm{P}}

\newcommand{\Erw}{\mathrm{E}}
\renewcommand{\d}{\,\mathrm{d}}
\newcommand{\pb}[1]{\parbox[0pt][#1][c]{0cm}{}}
\renewcommand{\vec}[1]{\hat{#1}}
\newcommand{\del}{\partial}
\newcommand{\supp}{\mathop{\mathrm{supp}}}

\newcommand{\eps}{\varepsilon}
\newcommand{\La}{\mathcal{L}_\w}
\newcommand{\Laeps}{\La^\eps}

\newcommand{\Lz}{\mathcal{L}^0}
\newcommand{\Pa}{\mathcal{P}}
\newcommand{\Pavt}{\Pa^{\w,V}_t}
\newcommand{\ue}{u^\eps}

\newcommand{\Z}{\mathbb{Z}}
\newcommand{\Zd}{\Z^{d}}
\newcommand{\Zdeps}{\Z^d_\eps}
\newcommand{\R}{\mathbb{R}}
\newcommand{\Rd}{\R^{d}}
\newcommand{\N}{\mathbb{N}}
\newcommand{\Hz}{{\mathcal{H}_0}}
\newcommand{\Heps}{{\mathcal{H}_\eps}}

\newcommand{\Reps}{\mathcal{R}_\eps}
\newcommand{\Repsadj}{\Reps^\ast}

\newcommand{\bzeps}{b\left( z, \frac{\eps}{2} \right)}
\newcommand{\Bz}{\mathcal{B}_0}
\newcommand{\Beps}{\mathcal{B}_\eps}
\newcommand{\E}{\mathbb{E}}
\newcommand{\Leb}{\mathscr{L}}

\newcommand{\edges}{E}
\newcommand{\w}{\omega}

\newcommand{\weakto}{\rightharpoonup}

\renewcommand{\P}{\mathbb{P}}

\newcommand{\rmw}{{\rm w}}
\newcommand{\cgu}{\nu} 

\newcommand{\ldef}{:=}

\newcommand{\mD}{\ensuremath{\mathrm{D}}}
\newcommand{\cO}{\ensuremath{\mathcal O}}

\newcommand{\Norm}[2]{%
  \ensuremath{%
    \mathchoice{\big\lVert #1 \big\rVert}
     {\lVert #1 \rVert}
     {\lVert #1 \rVert}
     {\lVert #1 \rVert}_{\raisebox{-.0ex}{$\scriptstyle #2$}}
  }
}

\begin{document}
  \DeclareGraphicsExtensions{.pdf}
  \numberwithin{equation}{section}

  \title[Homogenization RCM]{Homogenization theory for the random conductance model with degenerate ergodic weights and unbounded-range jumps}
  
  \author{Franziska Flegel}
\address{Weierstrass Institut f\"ur Angewandte Analysis und Stochastik (WIAS)}
\curraddr{Mohrenstr. 39, 10117 Berlin}
\email{flegel@wias-berlin.de}
\thanks{}
 
\author{Martin Heida}
\address{Weierstrass Institut f\"ur Angewandte Analysis und Stochastik (WIAS)}
\curraddr{Mohrenstrasse 39, 10117 Berlin}
\email{heida@wias-berlin.de}
\thanks{}
 
\author{Martin Slowik}
\address{Technische Universit\"at Berlin}
\curraddr{Strasse des 17. Juni 136, 10623 Berlin}
\email{slowik@math.tu-berlin.de}
\thanks{}

 
 
\date{\today}
 
\dedicatory{}

\begin{abstract}
  We study homogenization properties of the discrete Laplace operator with random conductances on a large domain in $\Zd$.
  More precisely, we prove almost-sure homogenization of the discrete Poisson equation and of the top of the Dirichlet spectrum.
  
  We assume that the conductances are stationary, ergodic and nearest-neighbor conductances are positive.
  In contrast to earlier results, we do not require uniform ellipticity but certain integrability conditions on the lower and upper tails of the conductances.
  We further allow jumps of arbitrary length.
  
  Without the long-range connections, the integrability condition on the lower tail is optimal for spectral homogenization.
  It coincides with a necessary condition for the validity of a local central limit theorem for the random walk among random conductances.
  As an application of spectral homogenization, we prove a quenched large deviation principle for the normalized and rescaled local times of the random walk in a growing box.
  
  Our proofs are based on a compactness result for the Laplacian's Dirichlet energy, Poincar\'e inequalities, Moser iteration and two-scale convergence.
\end{abstract}

\newgeometry{text={32pc,670pt},centering}
\maketitle

\tableofcontents
\restoregeometry

\section{Introduction}
Disordered media may homogenize in various ways.
For example, in a microscopically inhomogeneous material, the solution to the heat equation might satisfy a local limit theorem when viewed on larger scales.
This might be on the entire space or within a bounded domain with certain boundary conditions.
In bounded domains we can furthermore ask whether the solution to a Poisson equation homogenizes as the domain's diameter, i.e., the macroscopic scale, grows to infinity.
Or, alternatively, we can wonder whether the Dirichlet spectrum of the associated Laplace operator converges in some sense or whether the occupation time measures of the corresponding diffusion fulfill a large deviation principle.
When we say that we let the macroscopic scale grow to infinity, this is always in comparison to a microscopic scale $\eps$, which might tend to zero instead.

Although all these aspects of homogenization are \emph{a priori} different, intuition suggests that they should somehow be related.
Especially if the associated Laplace operator is self-adjoint, i.e., it is the generator of a reversible random walk, then the homogenization of the Poisson equation is strongly linked to spectral homogenization (see \cite[Chapter 11]{JKO1994}). Spectral homogenization in turn is linked to the validity of a large deviation principle for the occupation time measures of a random walk in bounded domains \cite[Theorem 5]{DV75}.
It is therefore plausible that these aspects of homogenization should hold under similar conditions.

For self-adjoint Laplace operators, a crucial condition for many kinds of asymptotic homogenization is -- apart from ergodicity -- the validity of a Poincar{\'e} inequality with an optimal constant independent of $\eps$ (uniform Poincar\'e inequality, see also \eqref{equ:PI}).
For spectral homogenization this is immediately evident since the optimal Poincar{\'e} constant is the inverse of the principal Dirichlet eigenvalue of the Laplacian (see Remark \ref{rem:optimality}).
In the situation of the present paper, we will see that the uniform Poincar{\'e} inequality is necessary and carries us quite far, although it is not completely sufficient for our results.
However, it leads us to conditions that are optimal (up to a critical case, see Remark \ref{rem:optimality}).

In the present paper, we examine a discrete disordered medium that belongs to a class of random conductance models on the lattice $\Zd$ with stationary and ergodic conductances on nearest-neighbor and unbounded-range connections.
Random conductance models are of high mathematical and physical interest (see \cite{Biskup2011review, Bouchaud1990} and references therein).
For these models, \cite[Theorem 1.11, Remark 1.12]{ADS16} already used the Poincar{\'e} inequality and a related Sobolev inequality to prove the validity of a local limit theorem for the heat kernel in the case where only nearest neighbors are connected.
As we explained above, a uniform Poincar{\'e} inequality is also necessary for spectral homogenization.
In this model, its validity depends on the integrability of the tails of the conductances (see e.g.\ Proposition \ref{prop:localPoincare} and \cite[Proposition 2.1]{ADS16}).
To be more precise, let $\w(e)$ denote the random conductance on the edge $e$ and define
\begin{align*}
  q\;=\;\sup\{ r\;\colon\, \E [\w(e)^{-r}]<\infty \}
\end{align*}
where $\E$ denotes the expectation with respect to $\w(e)$.
For the moment, let us assume that only nearest neighbors are connected, or equivalently, that only nearest-neighbor conductances carry a positive conductance.
Then the crucial assumption is
\begin{align}\label{equ:LowerMomentIntro}
  q
  \;>\;
  q_{\rm c}
  \;=\;
  \begin{cases}
    d/2, &\text{ for general stationary, ergodic conductances and $d\geq 2$,}\\
    1/4, &\text{ for i.i.d.\ conductances and $d\geq 2$,}\\
    1, &\text{ for $d=1$, }
  \end{cases}
\end{align}
see Assumption \ref{ass:Integ}.
Additionally we require that $\E [\w(e)]<\infty$ (cf.\ Assumption \ref{ass:Environ}\ref{ass:momcond}).
As we explain in Remark \ref{rem:optimality}, up to the critical case $q=q_{\rm c}$, the condition in \eqref{equ:LowerMomentIntro} is optimal.
In fact, $\E [\w(e)^{-q_{\rm c}}]<\infty$ is sufficient for the Poincar{\'e} inequality but not for the Moser iteration, which we use in Section \ref{subsec:MaximalInequalities}.
If $q<q_{\rm c}$, then it is possible (and in the i.i.d.\ case even almost sure \cite{Flegel2016}) that trapping structures as in Figure \ref{fig:trap} appear, which immediately contradict a uniform Poincar{\'e} inequality.

In addition to Poincar\'e and Sobolev inequalities, our proofs rely on stochastic two-scale convergence, an analytic method that is based on the ergodic theorem and was introduced in \cite{zhikov2006homogenization}.

A related problem of a nonlocal operator was recently studied by Piatnitski and Zhizhina \cite{piatnitski2017periodic} in the periodic case, where, as in the present article, the limit operator localizes to a second order elliptic operator.

For the random conductance model with conductances $\w(e)\in [0,c_0]$ ($c_0<\infty$) restricted to nearest-neighbor connections, Faggionato \cite{Faggionato2008} already used stochastic two-scale convergence in order to prove homogenization of the Laplace operator with a spectral shift on the infinite connected component of $\Zd$.
The spectral shift compensates for the lack of a Poincar\'e inequality. The very first successful application of two-scale convergence in the random conductance model seems to be by Mathieu and Piatnitski \cite{MP07}.

As a consequence of the homogenization of the Poisson equation on bounded domains, the homogenization of the top of the Dirichlet spectrum follows by the methods of \cite[Chapter 11]{JKO1994}, see Theorem \ref{thm:spectrum}.
For this result, Remark \ref{rem:optimality} explains in which sense Condition \eqref{equ:LowerMomentIntro} is sharp.
For i.i.d.\ conductances \eqref{equ:LowerMomentIntro} even decides between a completely homogenizing phase, which we cover in this paper, and a completely localizing phase of the principal Dirichlet eigenvector, which was studied in \cite{Flegel2016}.
We thus extend the results of Faggionato \cite{Faggionato2012} and Boivin and Depauw \cite{BoivinDepauw2003}.
Faggionato showed spectral homogenization in dimension $d=1$ under Condition \eqref{equ:LowerMomentIntro} and $\E [\w(e)]<\infty$,
whereas Boivin and Depauw proved spectral homogenization for conductances that are uniformly bounded from above and away from zero (uniform ellipticity).

As an application of the spectral homogenization, we prove a quenched large deviation principle for the occupation time measures, given that the random walk stays in a slowly growing box, see Proposition \ref{prop:LDP}.
We extend the results of \cite[Theorem 1.8]{KoenigWolff}, where the authors use the deep connection between the Dirichlet energy of the Laplace operator and the Donsker-Varadhan rate function of the occupation time measures of the associated random walk.

In the recent paper \cite{Neukamm2017}, the authors prove that under the same integrability conditions as in the present paper, the Dirichlet energy of the random conductance Laplacian $\Gamma$-converges to a deterministic, homogeneous integral (see their Corollary 3.4 and Proposition 3.24).
Together with their compactness result \cite[Lemma 3.14]{Neukamm2017}, Theorem 13.5 of \cite{DalMaso1993} implies the homogenization of the Poisson equation for models where the connections are of finite range.
With the method employed in the present paper, however, we can allow for unbounded-range connections (see Assumption \ref{ass:Environ}\ref{ass:momcond}). 
Further, we identify the corresponding limit operator in the long-range case (see Theorem \ref{thm:twoscale-strong})
and we prove the two-scale convergence of the gradient of the solution to the Poisson equation. 
Note that in the case of bounded-range connections, the result of \cite{Neukamm2017} together with the compact embedding of our Section \ref{sec:Compactness} implies that Assumption \ref{ass:Integ}\ref{ass:Integ:dhalf} is sufficient instead of Assumption \ref{ass:Integ}\ref{ass:Integ:q} for the homogenization result, see Remark \ref{rem:critcase}.

\subsection{Model and notation}
\label{sec:model}
Let us consider a graph with vertex set $\Zd$ and edge set
\begin{align}
 \edges \;=\; \{ \{ x,y \}:\, x,y\in \Z^d \text{ and } x\neq y \}\, ,
 \label{equ:DefEdges}
\end{align}
i.e., we assume that there exists an undirected edge between any two sites $x,y\in\Z^d$. 
We further assume that each edge $e$ carries a nonnegative random variable $\w(e)$, which we call the \emph{conductance} of the edge $e$.
Moreover, we call the family $\w = (\w(e))_{e\in\edges}$  \emph{environment} or \emph{landscape}.
If $e = \{ x,x+z \}$ for $x,z\in \Z^d$ ($z\neq 0$), we will also write $\w_{x,z}$ instead of $\w(e)$.

Moreover, $\tau_x$ denotes the translations by a vector $x\in\Z^d$, i.e., we write
\begin{align*}
  \w_{x,z} \;=\; (\tau_x \w)_{0,z}\, .
\end{align*}
Since the edges in \eqref{equ:DefEdges} are undirected, we have $\w_{x,z} = \w_{x+z,-z}$.

In the variable-speed random conductance model, we consider the Laplacian $\La$, which acts on real-valued functions $f\in\ell^2 (\Zd)$ as
\begin{align}
  (\La f)(x) \;=\;\sum_{z\in\Z^d} \w_{x,z}(f(x+z) - f(x))\, , \qquad \left( x\in\Z^d\right)\, .
  \label{equ:DefGenerator}
\end{align}
Note that $\La$ is \Pas well-defined under Assumption \ref{ass:Environ}\ref{ass:momcond}.
Since the conductances are symmetric, i.e., $\w_{x,z} = \w_{x+z,-z}$, the associated Markov process is reversible.
As we explain in Section \ref{sec:discrder}, the Laplacian $\La$ is the discrete analogue of a divergence-form operator with random weights.

We denote the probability space that governs the environment $\w$ by
\begin{align}
  (\Omega,\mathcal{F},\Pr)\;=\;\left( [0,\infty]^E , \mathcal{B} ( [0,\infty] )^{\otimes E }, \Pr \right)\, ,
  \label{eq:def-prob-space}
\end{align}
and the expectation with respect to the law $\Pr$ by $\E$.

For any $\omega\in\Omega$, we denote the set of \emph{open edges} by
\begin{align*}
  \cO
  \;\equiv\;
  \cO(\w)
  \;\ldef\;
  \{e \in \edges \,:\, \w(e) > 0\}
  \;\subset\;
  E.
\end{align*}
Further, $E_d \subset\edges$ denotes the set of all undirected nearest-neighbor bonds.

In this paper we will usually assume that the law $\Pr$ fulfills the following conditions.
\begin{ass}
  \label{ass:Environ}
  ~
  \begin{conditions}
    \item\label{ass:erg} The law $\Pr$ is stationary and ergodic with respect to spatial translations $(\tau_x)_{x\in\Z^d}$.
    \item\label{ass:momcond} $\E\left[ \sum_{z\in\Z^d}\w_{0,z} |z|^2 \right] <\infty$.
    \item\label{ass:FullLattice} For \Pae $\w$, the set $\cO(\w)$ of open edges contains the set $E_d$ of nearest-neighbor edges of $\Zd$.
  \end{conditions}
\end{ass}

In addition to Assumption \ref{ass:Environ}, our main results rely on an integrability condition for the lower tails of the conductances, for which we need to define the notion of paths in $(\Zd, E_d)$.
A path of length $l$ between $x$ and $y$ in $(\Zd, E_d)$ is a sequence $(x_i : i = 0, \ldots, l)$ with the property that $x_0 = x$, $x_l = y$ and $\{x_i, x_{i+1}\} \in E_d$ for any $i = 0, \ldots, l-1$.
If $\gamma=(x_i : i = 0, \ldots, l)$ is a path and there exists $i\in\{ 1,\ldots,l-1\}$ such that $\{x_i,x_{i+1}\}=e$, then we use the shorthand notation $e\in\gamma$.

For any $e \in E_d$ and $\mathbb{N} \ni l < \infty$, let $\Gamma_l(e)$ be a collection of paths in $(\Zd, E_d)$ between the vertices of the edge $e$ with length at most $l$ such that no two paths in $\Gamma_l(e)$ share an edge.
We define the measures $\nu^{\w}$ and $\nu_l^{\w}$ on $\Zd$ by
\begin{align}
  \nu^{\w}(x)
  \;\ldef\;
  \sum_{e\in E_d\colon\, x\in e} \w(e)^{-1}
  \qquad \text{and} \qquad
  \nu_l^{\w}(x)
  \;\ldef\;
  \sum_{e\in E_d\colon\, x\in e} \w_l (e)^{-1},
  \label{eq:def:nu}
  \end{align}
where
\begin{align}
  \w_l (e)^{-1}
  \;\ldef\;
  \min_{\gamma \in \Gamma_l (e)} \sum_{e' \in \gamma} \w(e')^{-1}\, .
  \label{equ:Defwl}
\end{align}
We let $\gamma_l^{\rm opt} (e)$ denote the minimizer of the RHS of \eqref{equ:Defwl}.
For an example of how to choose $\Gamma_9$ reasonably for the nearest-neighbor lattice $(\Z^d, E_d)$ if the conductances are independent and identically distributed, see Figure \ref{fig:IndepPaths}.

\begin{ass}[Lower moment condition]
  \label{ass:Integ}
  If $d=1$, then $\E[1/\w(e)]<\infty$ for any $e\in E_d$.
  In addition, if $d\geq 2$, then
  \begin{conditions}
     \item\label{ass:Integ:dhalf} there exists $l\in\mathbb{N}$ such that $\E \left[ (\nu^\omega_l (0))^{d/2} \right]<\infty$.
     \item[\mylabel{ass:Integ:q}{(a')}] there exists $l\in\mathbb{N}$ and $q>d/2$ such that $\E \left[ (\nu^\omega_l (0))^q \right]<\infty$.
  \end{conditions}
\end{ass}

\begin{remark}\label{rem:critcase}
  Note that Assumption \ref{ass:Integ}\ref{ass:Integ:dhalf} is sufficient for the Poincar\'e inequalities (Section \ref{sec:PI}) and the compact embedding (Section \ref{sec:Compactness}).
  The only reason why we need Assumption \ref{ass:Integ}\ref{ass:Integ:q} is the Moser iteration in the proof of Proposition \ref{prop:linftyPoiss}, which we need for the Auxiliary Lemma \ref{lem:twoscale-weak}.
  In fact, if we would assume that the length of the connections was bounded, or in other words, there exists $R<\infty$ such that \Pas $\cO (\omega) \subseteq \{ e \colon |e|<R \} $, then the authors of \cite{Neukamm2017} proved $\Gamma$-convergence under Assumption \ref{ass:Integ}\ref{ass:Integ:dhalf}.
  Therefore the compact embedding of Section \ref{sec:Compactness} implies the homogenization result Theorem \ref{thm:twoscale-strong} and thus Assumption \ref{ass:Integ}\ref{ass:Integ:dhalf} is sufficient.
\end{remark}

\setcounter{wraprem}{\value{thm}}
\setlength\columnsep{20pt}
\begin{wrapfigure}{r}{3.3cm}
  \vspace{-.4cm}
  \includegraphics[width=\linewidth]{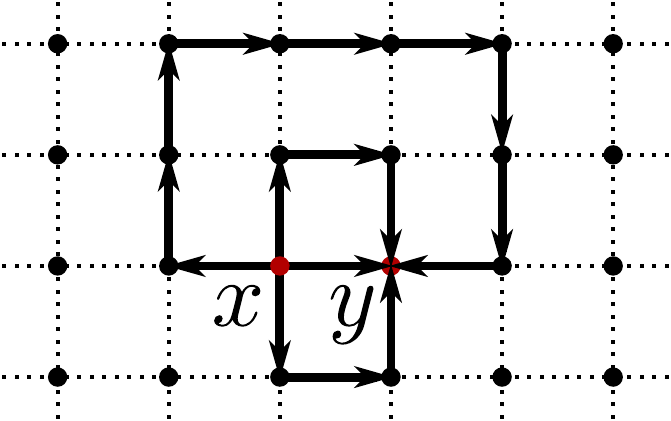}
  \caption{\\Independent paths\label{fig:IndepPaths}}
  \vspace{-0.2cm}
\end{wrapfigure}
\begin{wraprem}
  \label{rem:iid}
  Generally, $\E[\w(e)^{-d/2}]<\infty$ is sufficient for Assumption \ref{ass:Integ}\ref{ass:Integ:dhalf}.
  This can even be improved if the conductances $\w(e)$ $(e\in\edges)$ are independent and identically distributed (i.i.d.\@) and $d\geq 2$.
  For example, on the nearest-neighbor lattice $(\Zd,E_d)$ with independent and identically distributed conductances, Assumption \ref{ass:Integ}\ref{ass:Integ:dhalf} holds if $\E \bigl[ \w(e)^{-1/4} \bigr]<\infty$ for any edge $e\in E_d$.
  Similarly, Assumption \ref{ass:Integ}\ref{ass:Integ:q} holds if there exists $q>q_{\rm c}=1/4$ such that $\E \bigl[ \w(e)^{-q} \bigr]<\infty$ for any edge $e\in E_d$.
  This follows because any two sites in $\Z^d$ are connected through $2d$ independent nearest-neighbor paths (see Figure \ref{fig:IndepPaths}, cf. \cite[Fig. 2]{ADS16}, \cite[Fig. 2.1]{Kesten1984}).
  
  If we added further links to the set $E_d$, the number of independent paths between any two sites would increase whence the critical exponent $q_{\rm c}$ would decrease.
  If we assumed that $E_d$ would contain all the links of $E$, then it would even be sufficient to assume that there exists $q>0$ such that $\E \bigl[ \w(e)^{-q} \bigr]<\infty$.
  Note that in order not to violate Assumption \ref{ass:Environ}\ref{ass:momcond}, we would assume in this case that $\w(e)=\tilde{\w}(e)/|e|^{\alpha}$ where the $(\tilde{\w}(e))_e$ are i.i.d.\@, $\alpha>d+2$ and $|e|$ is the euclidean length of the edge $e$.
\end{wraprem}
\setcounter{thm}{\value{wraprem}}

\subsection[The rescaled lattice \texorpdfstring{$\Z^d_\eps$}{Zdeps}]{The rescaled lattice \texorpdfstring{$\boldsymbol{\Z^d_\eps}$}{Zdeps}}
\label{sec:rescaledZd}
We aim to consider the behavior of the operator $\La$ in boxes of the form $Q_n := (-n,n)^d \cap \Zd$ with zero Dirichlet boundary conditions.
More precisely, we fix an environment $\w$ on the entire $\Zd$, let the box size $n$ grow to infinity and want to characterize the behavior of solutions to the Poisson equation and the spectral problem.
For this purpose we use analytic techniques as introduced in Section \ref{sec:AnalyticTools}.
Regarding these techniques, it is more natural to replace the lattice $\Zd$ by the rescaled lattice $\Zdeps\ldef\eps\Zd$ and the growing box $Q_n$ by the box $Q_\eps \ldef Q\cap \Zdeps$ with $Q=(-1,1)^d$ and $\eps=n^{-1}$.

In this context, the Laplacian defined in \eqref{equ:DefGenerator} corresponds to the accelerated operator $\Laeps$ which acts on real-valued functions $f\in\ell^2 (\Zdeps)$ as
\begin{align}
  (\Laeps f) (x) &\;=\; \eps^{-2}\sum_{z\in\Zd} \w_{\frac{x}{\eps},z}\left[f\left(x+\eps z\right) - f\left(x\right)\right]\, , \qquad \left( x\in\Zdeps\right)\, ,
  \label{equ:DefGeneratorEps}
\end{align}
where the conductances $\w_{\frac{x}{\eps},z}$ remain random variables associated with the links in the edge set $\edges$, i.e., the links between sites in $\Z^d$.
Note that if $\La$ is the generator of a Markov process $(X_t)_{t\geq 0}$, then $\Laeps$ is the generator of the diffusively rescaled Markov process $(X_t^\eps)_{t\geq 0}$, which fulfills $X_t^\eps = \eps X_{\eps^{-2} t}$.

For $\eps, p>0$ and $A_\eps\subseteq \Zdeps$, we define the function spaces
\begin{align}
  \ell^p_\eps (A_\eps) \ldef \Bigl\{ v\;\colon\, \Zdeps \to \R \colon\,\mspace{3mu} \eps^{d}\mspace{-3mu}\sum_{x\in A_\eps}v(x)^p<\infty \Bigr\}
  \quad\text{with }
  \| v \|_{\ell^p_\eps (A_\eps)} \ldef \biggl( \eps^{d}\mspace{-3mu}\sum_{x\in A_\eps}v(x)^p \biggr)^{1/p}\, .
  \label{equ:DefEllpeps}
\end{align}
We abbreviate $\ell^p_\eps \ldef \ell^p_\eps (\Zdeps)$.

Analogously to $\ell^p_\eps$, we introduce the Hilbert spaces $\Hz, \Heps$ through
\begin{align*}
   \Hz &\;=\; \left\{ v\in L^2 \left(\mathbb{R}^d \right) \;\colon\, \supp v \subseteq Q\right\},\quad
   \Heps \;=\; \left\{ v\in\ell^2_\eps (\Zdeps) \;\colon\,  \supp v \subseteq Q_\eps\right\}
\end{align*}
and let $\Hz$ and $\Heps$ be equipped with the scalar products
\begin{align*}
\quad \langle u, v \rangle_{\Hz} &\;=\; \int_{\R^d} u(x)v(x) \d x\, , \quad\langle u^\eps, v^\eps \rangle_{\Heps} \;=\; \eps^d \sum_{z\in\Zdeps} u^\eps (z)\, v^\eps (z)\,.
\end{align*}
For $z\in\Zdeps$, we let $b(z,\eps/2)$ denote the half-open ball $z+ \left( -\eps/2, \eps/2 \right]^d$.
We define the local averaging operator $\Reps\colon\, \Hz \to \Heps$ acting on functions $f\in\Hz$ by
\begin{align}
	\left(\Reps f\right)(z) \;=\; \eps^{-d} \,\,\,\, \int_{\bzeps} f(x) \d x\quad z\in\Zdeps\, .
	\label{equ:Reps}
\end{align}
A direct calculation shows that its adjoint operator $\Repsadj\colon\, \Heps \to \Hz$ is given by
\begin{align}
	\Repsadj v^\eps \;=\; \sum_{z\,\in\, \Zdeps} v^\eps (z) \mathds{1}_{b\left( z,\frac{\eps}{2}\right)}\qquad (v^\eps\in \Heps)\, ,
	\label{equ:Repsadj}
\end{align}
where we write $\mathds{1}_{b\left( z,\frac{\eps}{2}\right)}$ for the characteristic function of $b\left( z,\frac{\eps}{2}\right)$.

\section{Main results}
\label{sec:mainresults}
\subsection{Homogenization}
Given a function $f^\eps\colon\, \Zdeps\to \R$, we are interested in the solution $u^\eps\in \Heps$ of the Poisson problem
\begin{align}
	-\Laeps u^\eps \;=\; f^\eps \qquad\text{on }Q_\eps
	\label{equ:epsPoisson}
\end{align}
with zero Dirichlet conditions.
The above problem has a unique solution because $-\Laeps$ is invertible on $\Heps$.

\begin{thm}
\label{thm:twoscale-strong}
Let $f^{\eps}\colon Q_\eps \to\R$ be a sequence of functions such that $\Repsadj f^{\eps}\weakto f$ weakly in $L^{2}(Q)$ for some $f\in L^{2}(Q)$.
If Assumptions \ref{ass:Environ} and \ref{ass:Integ}\ref{ass:Integ:q} hold, then for almost all $\omega\in\Omega$ the sequence of solutions $\ue\in\Heps$ to the problem \eqref{equ:epsPoisson} satisfies $\Repsadj \ue\to u$ strongly in $L^{2}(Q)$,
where $u\in H_{0}^{1}(Q)\cap H^{2}(Q)$ solves the limit problem
\begin{equation}
-\nabla\cdot\left(A_{\mathrm{hom}}\nabla u\right)=2f\,,
\label{eq:strong-limit-equation}
\end{equation}
almost everywhere in $Q$ with $A_{\hom}$ defined through \eqref{eq:Definition-A-hom}.
\end{thm}
We prove this theorem at the end of Section \ref{ref:ProofPoisson}.
In Lemma \ref{lem:non-degeneracy} we prove that $A_{\hom}$ is strictly positive definite and by standard arguments $A_{\hom}$ is symmetric.

Based on Theorem \ref{thm:twoscale-strong}, we introduce the operator 
$$\forall u \in L^2 (Q)\qquad \Lz u:=\nabla\cdot\left(A_{\hom}\nabla u\right)\,,$$
such that $-\Lz$ is symmetric positive definite operator on $L^2(Q)$ with domain $H^2(Q)$. 
\setcounter{wraprem}{\value{thm}}
\setlength\columnsep{20pt}
\begin{wrapfigure}{r}{2.5cm}
  \vspace{.2cm}
  \includegraphics[width=\linewidth]{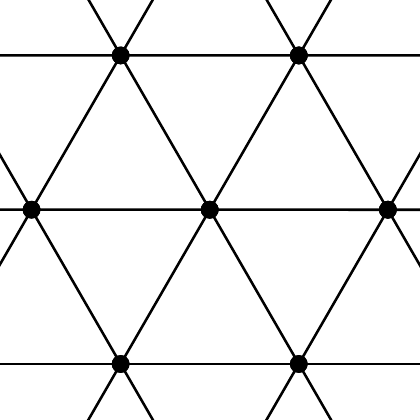}
  \caption{\\Triangular\\lattice.\label{fig:Triangular}}
  \vspace{-0.3cm}
\end{wrapfigure}
\begin{wraprem}
  \label{rem:OtherLattices}
  With our methods, Theorem \ref{thm:twoscale-strong} can be easily generalized for other lattices than $\Zd$.
  In order to apply our methods directly, we just have to require that the lattice is translationally invariant (for the two-scale convergence, see Section \ref{sec:two-scale}) and fulfills a Sobolev inequality (as in \eqref{eq:Sobolev:d=1} or \eqref{eq:Sobolev:Sa}) with isoperimetric dimension $d_{\rm ISO}$ (to obtain the necessary Poincar{\'e} inequalities and make the Moser iteration work).
  For example, the triangular lattice in Figure \ref{fig:Triangular} is translationally invariant and has isoperimetric dimension $d_{\rm ISO}=2$.
  If we therefore replace $\Zd$ by the triangular lattice and the dimension $d$ in Assumption \ref{ass:Integ} by the isoperimetric dimension $d_{\rm ISO}$, Theorem \ref{thm:twoscale-strong} still holds.
  
  Note that in view of Remark \ref{rem:iid}, we observe that in the case of i.i.d.\ conductances on the triangular lattice, Assumption \ref{ass:Integ}\ref{ass:Integ:dhalf} holds if $\E [\w(e)^{-1/6}]<\infty$.
\end{wraprem}
\setcounter{thm}{\value{wraprem}}

\begin{remark}
  \label{rem:perc}
  Although we focus here on the random conductance model with long-range jumps and positive nearest-neighbor conductances, our arguments do not require the full strength of this assumption.
  For instance, we can also extend the homogenization result to the \emph{nearest-neighbor} percolation case.
  More precisely, we can relax Assumption \ref{ass:Environ}\ref{ass:FullLattice} such that the set of open edges $\mathcal{O}(\w) \subset E_d$ forms a unique infinite cluster that satisfies both a volume regularity condition and a (weak) relative isoperimetric inequality on large scales, cf.\ \cite{DNS18}.
  Notice that in the nearest-neighbor percolation setting, similar homogenization results have also been obtained by Faggionato in \cite{Faggionato2008} under the additional assumption that the conductances are bounded from above.
\end{remark}

In order to infer the large deviation principle Proposition \ref{prop:LDP}, let us now consider the spectrum of the operators $-\Laeps + \Reps V$ with an arbitrary bounded, continuous potential $V\colon \Rd \to \R$.
On the domain $Q_\eps$ with zero Dirichlet conditions we can represent $-\Laeps + \Reps V$ as a real symmetric matrix and therefore we can choose the set $\bigl\{ \psi_j^{\eps}\bigr\}_{j=1,\ldots,k}$ of Dirichlet eigenvectors such that they form an orthonormal system.
By virtue of the Perron-Frobenius theorem (see e.g.\ \cite[Chapter 1]{SenetaNN}) the principal Dirichlet eigenvalue $\lambda_1^{\eps}$ is unique.
Thus, we now consider the problem
\begin{align}
\begin{aligned}
	&\psi_k^\eps \in \Heps, \quad \left(-\Laeps + \Reps V\right) \psi_k^\eps = \lambda_k^\eps \psi_k^\eps, \quad k = 1,2,\ldots\, ,\pb{1.5em}\\
	&\lambda_1^\eps < \lambda_2^\eps \leq \ldots \leq \lambda_k^\eps \ldots \,  ,\pb{1.5em}\\
	&\langle \psi_k^\eps, \psi_l^\eps \rangle_{\Heps} = \delta_{kl}\, .\pb{1.5em}
	\label{equ:DefEigvecEps}
\end{aligned}
\end{align}
Similarly, we consider the spectrum of the operator $\Lz$, i.e.,
\begin{align}
\begin{aligned}
	&\psi_k^0 \in \Hz, \quad \left( -\Lz +V\right)\psi_k^0 = \lambda_k^0 \psi_k^0, \quad k = 1,2,\ldots\, ,\pb{1.5em}\\
	&\lambda_1^0 < \lambda_2^0 \leq \ldots \leq \lambda_k^0\ldots \, ,\pb{1.5em}\\
	&\langle \psi_k^0, \psi_l^0 \rangle_{\Heps} = \delta_{kl}\, .\pb{1.5em}
\end{aligned}
\end{align}

In order to study the homogenization of \eqref{equ:DefEigvecEps} with a non-trivial potential $V$, we need the following result.
\begin{prop}
\label{prop:twoscale-strong}
Let $f^{\eps}\colon Q_\eps\to\R$ be a sequence of functions such that $\Repsadj f^{\eps}\weakto f$ weakly in $L^{2}(Q)$ for some $f\in L^{2}(Q)$.
Let $V\colon \Rd \to \R$ be a bounded, continuous potential such that $\liminf_{\eps\to0}\lambda_1^\eps>0$.
If Assumptions \ref{ass:Environ} and \ref{ass:Integ}\ref{ass:Integ:q} hold, then for almost all $\omega\in\Omega$ the sequence of solutions $\ue\in\Heps$ to the problem 
\begin{align}
  \left(-\Laeps + \Reps V\right) \ue \;=\; f^\eps
  \label{eq:prop:twoscale-strong-eps}
\end{align}
satisfies $\Repsadj \ue\to u$ strongly in $L^{2}(Q)$, where $u\in H_{0}^{1}(Q)\cap H^{2}(Q)$ solves the limit problem
\begin{equation}
-\nabla\cdot\left(A_{\mathrm{hom}}\nabla u\right)+2Vu\;=\;2f\,,
\label{eq:prop:strong-limit-equation}
\end{equation}
almost everywhere in $Q$ with $A_{\hom}$ defined through \eqref{eq:Definition-A-hom}.
\end{prop}
We prove this proposition in Section \ref{sec:ProofSpectrum}.
Note that under Assumption \ref{ass:Integ}\ref{ass:Integ:q}, the condition $V\geq 0$ is sufficient for $\liminf_{\eps\to0}\lambda_1^\eps>0$.

By virtue of \cite[Lemma 11.3, Theorem 11.5]{JKO1994}, Proposition \ref{prop:twoscale-strong} implies the following result, see Section \ref{sec:ProofSpectrum}.
Note that for the spectral result we can drop the assumption $\liminf_{\eps\to0}\lambda_1^\eps>0$ as we explain in Section \ref{sec:ProofSpectrum}.
\begin{thm}
	\label{thm:spectrum}
	Let $V\colon \Rd \to \R$ be a bounded, continuous potential and let $k\in\N$.
	If Assumptions \ref{ass:Environ} and \ref{ass:Integ}\ref{ass:Integ:q} hold, then
	\begin{align}
		\lambda_k^\eps \to \lambda_k^0 \qquad \text{\Pas as } \eps\to 0\,.
	\end{align}
	Further, the following statements are true:
	\begin{enumerate}[label={(\roman*)},ref={\thethm~(\roman*)}]
	 \item\label{cor:spectrum:subseq} Let $k\in\N$ and let $\eps_m$ be a null sequence.
	 Then there \Pas exists a family $\{ \psi_j^0 \}_{1\leq j\leq k}$ of eigenvectors of the operator $-\Lz + V$ and a subsequence, still indexed by $\eps_{m}$, along which the vector
	\begin{align*}
	  \left( \mathcal{R}^\ast_{\eps_m}\psi_1^{\eps_m}, \ldots, \mathcal{R}^\ast_{\eps_m}\psi_k^{\eps_m} \right) \to \left( \psi_1^{0}, \ldots, \psi_k^{0} \right)\qquad \text{strongly in }L^2(Q)\, .
	\end{align*}
	\item\label{cor:spectrum:lincom} On the other hand, if the multiplicity of $\lambda_k^0$ is equal to $s$, i.e.,
	\begin{align*}
		\lambda_{k-1}^0 < \lambda_k^0 = \ldots = \lambda_{k+s-1}^0 < \lambda_{k+1}^0 \qquad (\text{with } \lambda_0^0 <\lambda_1^0\text{ arbitrary})\, ,
	\end{align*}
	then there \Pas exists a sequence $\psi^\eps \in \Heps$ such that
	\begin{align}
		\lim_{\eps\to 0} \| \psi^\eps - \Reps \psi_k^0 \|_{\Heps} = 0\, ,
	\end{align}
	where $\psi^\eps$ is a linear combination of the eigenfunctions of the operator $-\Laeps+\Reps V$
	corresponding to the eigenvalues $\lambda_k^\eps, \ldots, \lambda_{k+s-1}^\eps$.
	\end{enumerate}
\end{thm}
Note that Biskup, Fukushima and K\"onig \cite{BFK16} proved a spectral homogenization theorem for a random bounded potential and the standard lattice Laplacian.
They later extended their result to unbounded potentials in \cite{BFK17}.

\setcounter{wraprem}{\value{thm}}
\setlength\columnsep{20pt}
\begin{wraprem}
  \label{rem:optimality}
  Let us discuss in what sense Assumption \ref{ass:Integ} is optimal for the result of Theorem \ref{thm:spectrum} with $V=0$.
  Since the principal Dirichlet eigenvalue has the variational representation 
  \begin{align*}
    \lambda_1^\eps \;=\; \inf \left\{ \langle \ue, -\Laeps \ue \rangle_{\Heps} \;\colon\, \ue\in\Heps \text{ and } \|\ue\|_{\Heps}=1 \right\}
  \end{align*}
  (also known as the Rayleigh-Ritz formula, or the Courant-Fischer theorem), it is necessary for spectral homogenization that \Pas there exists $C<\infty$ such that
  \begin{align}
    \| \ue \|_{\Heps}^2 \;\leq\; C\langle \ue, -\Laeps \ue \rangle_{\Heps}\qquad\text{for all $\ue\in\Heps$}
    \label{equ:Poincare:rem:optimality}
  \end{align}
  and for all $\eps>0$ (uniform Poincar\'e inequality).
  
  \begin{wrapfigure}{r}{3cm}
    \includegraphics[width=\linewidth]{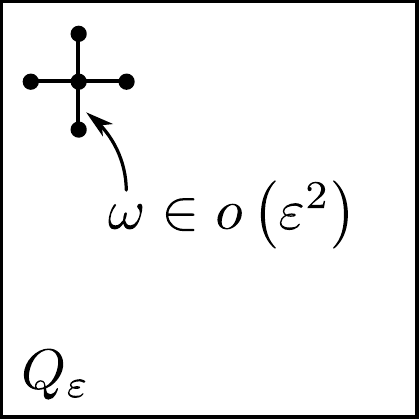}
    \caption{\\Variable-speed trap in $d\geq 2$.\label{fig:trap}}
    \vspace{-.2cm}
  \end{wrapfigure}
  
  If we assume that \Pas only nearest-neighbor connections carry a positive conductance, i.e., $\cO (\w) = E_d$, then Assumption \ref{ass:Integ} is optimal for the uniform Poincar\'e inequality up to the critical case $\sup\{ r\colon \E [\w(e)^{-r}]<\infty \}=q_{\rm c}$ (cf. \eqref{equ:LowerMomentIntro}).
  This means that if $\sup\{ r\colon \E [\w(e)^{-r}]<\infty \}<q_{\rm c}$, then it is possible to construct an environment where the uniform Poincar\'e inequality does not hold as $\eps$ tends to zero.
  
  For $d\geq 2$, this is due to trapping structures as in Figure \ref{fig:trap} where $\ue$ can concentrate its entire mass, see e.g.\ \cite[Section 1.4]{Flegel2016}.
  The construction of stationary, ergodic environments with such trapping structures is analogous to the one of a trap for the constant-speed random walk in \cite[Theorem 5.4]{ADS16}.
  In the i.i.d.\ case and if $\sup\{ r\colon \E [\w(e)^{-r}]<\infty \}<1/4$, the traps occur even \Pas for $\eps$ small enough and the principal Dirichlet eigenvector localizes \Pas in a single site \cite[Theorem 1.8]{Flegel2016}.
  
  In $d=1$ and if $\sup\{ r\colon \E [\w(e)^{-r}]<\infty \}<1$, even an i.i.d.\ environment contradicts the uniform Poincar\'e inequality:
  By a Borel Cantelli argument we can show that \Pas for $\eps$ small enough there exist edges $e_1=\{x_1,y_1\}$ and $e_2=\{x_2, y_2 \}$ such that $x_1\in (-\eps^{-1},-\eps^{-1}/2)\cap \Z$ and $x_2\in(\eps^{-1}/2,\eps^{-1})\cap \Z$, respectively, and such that both $\w(e_1)$ and $\w(e_2)$ decay much faster than $\eps$.
  When we insert a function $\ue\in\Heps$ into \eqref{equ:Poincare:rem:optimality} that is 1 on the interval $[\max (\eps x_1,\eps y_1), \min (\eps x_2,\eps y_2)]$ and zero otherwise, then we see that $C$ diverges as $\eps$ tends to zero, which is a contradiction to a uniform Poincar\'e inequality.
  
  If we assume that $\cO (\w)$ is \Pas strictly larger that $E_d$ but contains only connections of bounded length, an analogous construction as in \cite[Theorem 5.4]{ADS16} shows that $q_{\rm c}=d/2$ is still optimal in the general stationary ergodic case with $d\geq 2$.  
  For independent conductances however, $q_{\rm c}$ decreases when the upper bound for the length of the connections increases, see also Remark \ref{rem:iid}.
  On the other hand, if we assume that $\cO (\w)$ contains connections of unbounded length, all the suggested counterexamples fail and the question about the optimal conditions requires further research.
\end{wraprem}
\setcounter{thm}{\value{wraprem}}

\subsection{Local times of the random walk among random conductances}
For a fixed realization $\w$ of the environment, we consider the Markov process ${(X_t\,:\;t \geq 0})$ on $\Z^d$, which jumps with rate $\w_{x,z}$ from a site $x$ to the site $x+z$.  Since the holding times are site-dependent, this Markov process is called the variable-speed random walk among random conductances (see \cite{Biskup2011review} for a review).  Its generator $\La$ is given by \eqref{equ:DefGenerator}.
We let $\Prw^\w_x$ denote the law of a random walk that starts in site $x$ at time zero.

Our main motivation for this paper is to prove a quenched large deviation principle (LDP) for the \emph{occupation time measures} or \emph{local times}
\begin{align}
  l_t (z) \;\ldef\; \int_0^t \delta_{X_s} (z)\, \d s\qquad (z\in\Zd, t>0)
\end{align}
of the random walk among random conductances, given that the random walk stays in a certain growing region of the lattice.
More precisely, we define a spatial scaling $\alpha_t$ with $1 \ll \alpha_t \ll \sqrt{t}$ and consider the rescaled local times
\begin{align}
  L_t (z)
  \;\ldef\;
  \frac{\alpha_t^d}{t}\, l_t (\lfloor \alpha_t z \rfloor) 
  \qquad (z\in\Rd, t>0)\, .
\end{align}
Further, let $Q=(-1,1)^d$ and define $Q_t = \alpha_t Q \cap \Zd$.
In \cite[Theorem 1.8]{KoenigWolff}, the authors prove a quenched large deviation principle for the function $L_t$ given that $\supp (l_t) \subset Q_t$ and under the assumption that the conductances are i.i.d.\ and uniformly elliptic.
Our aim is to generalize this result to stationary and ergodic conductances and replace the uniform ellipticity condition by a suitable moment condition.

Let us recall some facts about the local times of the simple random walk.
We define the set
\begin{align}
  \mathcal{F} \;=\; \left\{ f^2 \colon f \in L^2 (Q), \| f \|_2 =1 \right\}
\end{align}
and equip $\mathcal{F}$ with the weak topology of integrals against bounded continuous functions $V\colon Q \to \R$.
Notice that on the event $\{ \supp (l_t) \subset Q_t \}$ the function $L_t$ is an element of the set $\mathcal{F}$ and an $L^1$-normalized random step function on $\Rd$.

In the case of a simple random walk, i.e., when $\w_{x,z}\equiv 1$, it is known that on the event $\{ \supp (l_t) \subset Q_t \}$ the function $L_t$ satisfies a large deviation principle on $\mathcal{F}$ with scale $t \alpha_t^{-2}$ and rate function $I_0 = I^{\rm SRW} - \inf_{\mathcal{F}} I^{\rm SRW}$, where
\begin{align}
  I^{\text{SRW}} (f)
  \;=\;
  \begin{cases}
    \sum_{i=1}^d \int_Q \left( \del_i f(y) \right)^2 \d y = \| \nabla f \|_2^2,
    &f\in H^1_0 (Q)\, ,
    \\
    \infty\, , &\text{else,}
  \end{cases}
\end{align}
see \cite{KoenigWolff} for further explanation and \cite{GKS07}.
We prove that under quite general conditions, this is also true for the random conductance model, see Proposition \ref{prop:LDP} and Corollary \ref{cor:LDP}.
For general stationary and ergodic conductances, however, the resulting rate function reads
\begin{align}\label{equ:DefRateFct}
  I_0 = I - \inf_{\mathcal{F}} I\qquad\text{where }
  I (f)
  \;=\;
  \begin{cases}
    \int_Q (\nabla f) \cdot A_{\rm hom} \nabla f\, , &f\in H^1_0 (Q),
    \\
    \infty\, , &\text{else,}
  \end{cases}
\end{align}
and the matrix $A_{\rm hom}\in \R^d\times \R^d$ is defined as in \eqref{eq:Definition-A-hom}.
\begin{ass}[Heat kernel lower bounds] \label{ass:heatkernel}
  There exists $c>0$ such that \Pas for $t$ large enough
  \begin{align}
    \Prw^\w_0 \left[ X_{t} =x \right] \;\geq\; ct^{-d/2}
  \end{align}
  for all $x \in \Zd$ with $|x| \leq \sqrt{t}$.
\end{ass}
\begin{prop}\label{prop:LDP}
Let Assumptions \ref{ass:Environ}, \ref{ass:Integ}\ref{ass:Integ:q} and \ref{ass:heatkernel} be fulfilled.
Then \Pas the rescaled local times $L_t$ satisfy a large deviation principle with respect to the weak topology of integrals against bounded continuous functions $V\colon Q \to \R$ under $\mathrm{P}_0^\w \left[\,\cdot \mid \supp( l_t ) \subset \alpha_t Q \right]$ on $\mathcal{F}$.
The scale is $t \alpha_t^{-2}$ and the rate function $I_0$ is defined in \eqref{equ:DefRateFct}.
\end{prop}
We prove this proposition in Section \ref{sec:ProofPropLDP} as a consequence of Theorem \ref{thm:spectrum}.

In the special case where only nearest-neighbor conductances are positive, Proposition \ref{prop:LDP} together with the heat kernel bounds of \cite[Proposition 4.7]{ADS16} respectively, implies the following corollary.
\begin{cor}
  \label{cor:LDP}
  Let the conductances be stationary and ergodic with law $\P$ and let \Pas $\cO(\omega) = E_d$.  For $p, q \in [1, \infty]$ satisfying $1/p+1/q < 2/d$ assume that $\E[\w(e)^p] < \infty$ and $\E[\w(e)^{-q}] < \infty$ for any $e \in E_d$.
  Then the large deviation principle from Proposition \ref{prop:LDP} holds.
\end{cor}

\section{Inequalities}
\label{sec:inequalities}
In analogy to the definition of $\ell^p_\eps$ in \eqref{equ:DefEllpeps}, we define the following space-averaged norms for functions $f\colon\, \Zd\to\R$.
Let $A\subseteq \Zd$ be a non-empty set and $p\in[1,\infty)$. Then
\begin{align}
  \Norm{f}{p,A}
  \;\ldef\;
  \bigg(
    \frac{1}{|A|}\; \sum_{x \in A}\, |f(x)|^p
  \bigg)^{\!\!1/p}
  \qquad \text{and} \qquad
  \Norm{f}{\infty, A} \;\ldef\; \max_{x\in A} |f(x)|\, ,
\end{align}
where $|A|$ is the counting measure on $A$. Moreover, we let
\begin{align}
  (f)_A \;\ldef\; |A|^{-1} \sum_{x\in A} f(x)
  \label{equ:DefAverage}
\end{align}
abbreviate the average of $f$ over the set $A$.

\subsection{Poincar{\'e} and Sobolev inequalities}
\label{sec:PI}
The main objective in this subsection is to prove weighted Poincar\'e and Sobolev inequalities.
The Poincar\'e inequalities of Proposition \ref{prop:localPoincare} and \eqref{equ:PI} are the main tools in the proof of Lemma \ref{lem:precompact}, whereas the Sobolev inequality of Proposition \ref{prop:Sobolev} with $\rho>1$ ensures uniform $\ell^\infty$-bounds of the solution to the Poisson equation (see Section \ref{subsec:MaximalInequalities}).

Starting point for our further considerations is the fact that the underlying unweighted Euclidean lattice $(\Zd, E_d)$ satisfies the classical Sobolev inequality for any $d \geq 1$.
Let $B \subset \Zd$ be finite and connected and $u\!: \Zd \to \R$.  Then, 
\begin{align}\label{eq:Sobolev:d=1}
  \inf_{a \in \R}
  \Norm{u - a}{\infty, B}
  \;\leq\;
  C_1\, |B|^{1/d}\,
  \bigg(
    \frac{1}{|B|}\mspace{-6mu}
    \sum_{\substack{x, y \in B \\ \{x,y\} \in E_d}} \mspace{-8mu}
    \big| u(x) - u(y) \big|
  \bigg)
\end{align}
for $d=1$, whereas for any $d \geq 2$ and $\alpha \in [1, d)$ we have
\begin{align}\label{eq:Sobolev:Sa}
  \inf_{a \in \R}
  \Norm{u - a}{\frac{d \alpha}{d-\alpha}, B}
  \;\leq\;
  C_1\, |B|^{1/d}\,
  \bigg(
    \frac{1}{|B|}\mspace{-6mu}
    \sum_{\substack{x, y \in B \\ \{x, y\} \in E_d}} \mspace{-8mu}
    \big| u(x) - u(y) \big|^\alpha
  \bigg)^{\!\!1/\alpha}\mspace{-15mu}.
\end{align}
For $d \geq 2$ this Sobolev inequality follows from the isoperimetric inequality of the underlying Euclidean lattice, see e.g.\ \cite[Theorem~3.2.7]{Ku10}.
\begin{prop}[local Poincar\'e inequality]
  \label{prop:localPoincare}
  For any $x_0 \in \Zd$ and $n \geq 1$, let $B(n) \equiv B(x_0,n)\subset \Zd$.
  Suppose that $d=1$ and that $\nu^{\w} (x) <\infty$ for all $x\in\Zd$.
  Then, there exists $C_{\mathrm{PI}} < \infty$ such that
  \begin{align}\label{eq:Poincare:weighted:d=1}
    \Norm{u - (u)_{B(n)}}{2, B(n)}^2
    \;\leq\;
    C_{\mathrm{PI}}\, \Norm{\nu^{\w}}{1, B(n)}\, \frac{n^2}{|B(n)|}\mspace{-3mu}
    \sum_{x, y \in B(n)} \mspace{-6mu}
    \w(\{x,y\})\, \big| u(x) - u(y) \big|^2
  \end{align}
  for any $u\!:\Z \to \R$.
  
  Furthermore, for every $d\geq 2$ and $l \in [1,\infty)$ with $\nu_l^{\w} (x) <\infty$ for all $x\in\Zd$, there exist constants $C_{\mathrm{PI}} \equiv C_{\mathrm{PI}}(d, l) < \infty$ and $C_{\mathrm{W}} \equiv C_{\mathrm{W}}(l) < \infty$ with $C_{\mathrm{W}}(1)=1$ such that
  \begin{align}\label{eq:Poincare:weighted}
    &\Norm{u-(u)_{B(n)}}{2,B(n)}^2
    \nonumber\\[.5ex]
    &\mspace{36mu}\leq\;
    C_{\mathrm{PI}}\, \Norm{\nu_l^{\w}}{\frac{d}{2}, B(n)}\,
    \frac{n^2}{|B(n)|}\mspace{-3mu}
    \sum_{x, y \in B(C_{\mathrm{W}} n)} \mspace{-15mu}
    \w(\{x,y\})\, \big| u(x) - u(y) \big|^2,
  \end{align}
  for any $u\!:\Zd \to \R$, where the measure $\nu_l$ is given by \eqref{eq:def:nu} with suitable path sets $\Gamma_l$.
\end{prop}

\begin{proof}[Proof of Proposition \ref{prop:localPoincare}]
  As in \cite[Proposition~2.1 or 6.1]{ADS16}, the assertion is an immediate consequence of \eqref{eq:Sobolev:Sa} and H\"older's inequality (see also \cite[Lemma~2.3]{GM18}).  Nevertheless, we will repeat the argument here for the reader's convenience.

  Since $\Norm{u - (u)_{B(n)}}{2,B(n)} = \inf_{a \in \R} \Norm{u-a}{2,B(n)} \leq \inf_{a \in \R} \Norm{u-a}{\infty,B(n)}$,
  the assertion \eqref{eq:Poincare:weighted:d=1} follows from \eqref{eq:Sobolev:d=1} by an application of the Cauchy-Schwarz inequality.
  
  Let us now consider \eqref{eq:Poincare:weighted}, i.e., the case $d \geq 2$.
  For $e=\{x,y\}\in E_d$ we let $|\nabla u(e)|$ denote the difference $|u(x)-u(y)|$.
  For any $e\in E_d$ we observe that by the Cauchy-Schwarz inequality
  \begin{align*}
    |\nabla u(e)|
    \;\leq\;
    \bigg(\frac{1}{\w_l(e)}\bigg)^{\!\!1/2}\,
    \bigg(
      \sum_{e' \in E_d} \,\w(e')\, |\nabla u(e')|^2\,
      \mathds{1}_{e' \in \gamma_l^{\rm opt} (e)}
    \bigg)^{\!\!1/2}\mspace{-15mu},
  \end{align*}
  where we recall the definitions of $\w_l$ and $\gamma_l^{\rm opt}$ in \eqref{equ:Defwl} and below.
  Thus, for any $\alpha \in [1, 2)$, H\"older's inequality yields
  \begin{align}\label{eq:gradient:alpha}
    \bigg(
      \frac{1}{|B(n)|}\mspace{-3mu}
      \sum_{\substack{x,y \in B(n)\\\{x,y\} \in E_d}}\mspace{-5mu}
      |\nabla u(\{x,y\})|^{\alpha}
    \bigg)^{\!\!1/\alpha}
    \leq
    \Norm{\nu_l^{\w}}{\alpha/(2-\alpha), B(n)}^{1/2}\, 
    \bigg(
      \frac{1}{|B(n)|}
      \sum_{e' \in E_d}\mspace{-3mu} \w(e')\, |\nabla u(e')|^2\,
      N_l(e')
    \bigg)^{\!\!1/2}\mspace{-15mu},
  \end{align}
  where
  \begin{align*}
    N_l(e') \;\ldef\; \sum_{\substack{x,y \in B(n)\\\{x,y\} \in E_d}} \mathds{1}_{e' \in \gamma_l^{\rm opt} (\{ x,y \})}\quad\text{for any }e' \in E_d\, .
  \end{align*}
  Note that there exists $c<\infty$ such that $N_l(e') \leq c l^d$ for any $e' \in E_d$.
  In addition, there exists $C_{\rm W}<\infty$ such that $N_l ({x,y})=0$ if $x,y\notin B(C_{\rm W} n)$.
  Thus, when we choose $\alpha = 2d/(d+2)$, then \eqref{eq:Poincare:weighted} follows from \eqref{eq:Sobolev:Sa}.
\end{proof}

We define
\begin{align}
    \mathcal{E}_\w (u) \ldef \langle u,-\La u \rangle_{\ell^2 (\Zd)}\, ,
    \qquad(u\;\colon\, \Zd\to\R, u\in\ell^2 (\Zd))\, .
    \label{equ:DefDirichEn}
\end{align}

Our next task is to establish the corresponding versions of \eqref{eq:Sobolev:d=1} and \eqref{eq:Sobolev:Sa} on the weighted graph $(\Zd, \edges_d, \w)$.
For this purpose, for $d \geq 2$ and $q \geq 1$ we define
\begin{align}
  \rho
  \;\equiv\;
  \rho(d,q)
  \;\ldef\;
  \frac{d}{d-2+d/q}.
\end{align}
Notice that $\rho(d,q)$ is monotonically increasing in $q$ and converges to $d/(d-2)$ as $q$ tends to infinity.  Moreover, $\rho(d,d/2) = 1$.

\begin{prop}[Sobolev inequality]
\label{prop:Sobolev}
  Let $x_0 \in \Zd$ and $n \in\mathbb{N}$. 
  Suppose that $d=1$ and that $\nu^{\w} (x) <\infty$ for all $x\in\Zd$.
  Then there exists $C_{\mathrm{S}} < \infty$ such that
  \begin{align}\label{eq:Sobolev:weighted:d=1}
    \Norm{u^2}{\infty, B(x_0, n)}
    \;\leq\;
    C_{\mathrm{S}}\, n^2\, \Norm{\nu^{\w}}{1, B(x_0, n)}\, 
    \frac{\mathcal{E}_{\w}(u)}{|B(x_0, n)|}
  \end{align}
  for any $u\!:\Z \to \R$ with $\supp u \subset B(x_0, n)$.
  
  Furthermore, for every $d\geq 2$, $q \in [1, \infty)$ and $l \in [1,\infty)$ with $\nu_l^{\w} (x) <\infty$ for all $x\in\Zd$, there exists $C_{\mathrm{S}} \equiv C_{\mathrm{S}}(d, q, l) < \infty$ such that
  \begin{align}\label{eq:Sobolev:weighted}
    \Norm{u^2}{\rho, B(x_0, n)}
    \;\leq\;
    C_{\mathrm{S}}\, n^2\, \Norm{\nu_{l}^{\w}}{q, B(x_0, n)}\,
    \frac{\mathcal{E}_{\w}(u)}{|B(x_0, n)|}
  \end{align}
  for any $u\!:\Zd \to \R$ with $\supp u \subset B(x_0, n)$, where the measure $\nu_l$ is given by \eqref{eq:def:nu} with suitable path sets $\Gamma_l$.
  %
\end{prop}
In analogy to \eqref{equ:DefDirichEn}, we define
\begin{align}
    \mathcal{E}_\omega^\eps(u^\eps)\;\ldef\;\left\langle  u^\eps, -\Laeps u^\eps\right\rangle_{\Heps}\,.
    \label{equ:DefDirichEnEps}
\end{align}
We prove Proposition \ref{prop:Sobolev} after the following remark.
\begin{remark}\label{rem:PI}
  For $d=1$, Proposition \ref{prop:Sobolev} implies that
  \begin{align}
    \max_{x\in Q_\eps} \left( \ue (x)\right)^2 \;\leq\; C_{\rm S} \Norm{\nu^{\w}}{1, B_{1/\eps}}\,  \mathcal{E}_{\w}^\eps (\ue)\, .
    \label{equ:linfty1D}
  \end{align}
  For $d\geq 2$, Proposition \ref{prop:Sobolev} implies that
  \begin{align}
    \| (\ue)^2 \|_{\ell^\rho_\eps (Q_\eps)} \;\leq\; C_{\rm S} \Norm{\nu_l^{\w}}{q, B_{1/\eps}}\,  \mathcal{E}_{\w}^\eps (\ue)\, .
    \label{equ:SobolevEps}
  \end{align}
  When we insert $q=d/2$ into \eqref{equ:SobolevEps}, we especially obtain that
  \begin{align}
    \| \ue \|^2_{\ell^2_\eps (Q_\eps)} \;\leq\; C_{\rm S} \Norm{\nu_l^{\w}}{\frac{d}{2}, B_{1/\eps}}\,  \mathcal{E}_{\w}^\eps (\ue)\, .
    \label{equ:PoincareBddSuppEps}
  \end{align}
  Under Assumption \ref{ass:Integ}\ref{ass:Integ:dhalf} and by virtue of the ergodic theorem, \eqref{equ:PoincareBddSuppEps} and \eqref{eq:Sobolev:weighted:d=1} imply that for $d\geq 1$ there exists a \Pas finite $C(\omega)$ such that for all $\eps>0$ and all $\ue\in\Heps$ we have
  \begin{align}
    \| \ue \|^2_{\Heps} \leq C(\omega)\, \mathcal{E}_{\w}^\eps (\ue) \quad\text{(uniform Poincar{\'e} inequality)}\, .
    \label{equ:PI}
  \end{align}
\end{remark}

\begin{proof}[Proof of Proposition \ref{prop:Sobolev}.]
  In the sequel we will give a proof only for \eqref{eq:Sobolev:weighted}.
  The assertion \eqref{eq:Sobolev:weighted:d=1} follows by similar arguments.
  To lighten notation, set $B(n) \equiv B(x_0, n)$ and define $A(n) \ldef B(2n) \backslash B(n)$.
  The constant $c \in (0, \infty)$ appearing in the computations below is independent of $\alpha$ but may change from line to line.
  Let $a \in \R$ and $\alpha \in [1, d)$.
  Since $u(x)=0$ for $x\in A(n)$, we have
  \begin{align*}
    |a|
    \;=\;
    \frac{1}{|A(n)|}\, \sum_{x \in A(n)} |u(x) - a|
    \;\leq\;
    \frac{|B(2n)|}{|A(n)|}\, \Norm{u-a}{1,B(2n)}
    \;\leq\;
    c\, \Norm{u-a}{\frac{d\alpha}{d-\alpha}, B(2n)}.
  \end{align*}
  Hence, an application of Minkowski's inequality yields
  \begin{align*}
    \Norm{u}{\frac{d\alpha}{d-\alpha}, B(n)}
    \;\leq\;
    \Norm{u-a}{\frac{d\alpha}{d-\alpha}, B(n)} \,+\, |a|
    \;\leq\;
    c\, \Norm{u-a}{\frac{d\alpha}{d-\alpha}, B(2n)}.
  \end{align*}
  Thus, for any $q \geq 1$ the assertion \eqref{eq:Sobolev:weighted} follows as in the previous proof from \eqref{eq:Sobolev:Sa} combined with \eqref{eq:gradient:alpha} by choosing $\alpha = 2q/(q+1) \in [1, 2)$.
\end{proof}

\subsection{Maximal inequality}
\label{subsec:MaximalInequalities}

\begin{prop}[$\ell^\infty$-bound for solution of Poisson equation in $d\geq 2$]
  \label{prop:linftyPoiss}
  Let $d\geq2$ and suppose that $\ue:\, \Zdeps\to \R$ is a solution of \eqref{equ:epsPoisson}.
  For some fixed $l\in[1,\infty)$ consider the measure $\nu_l^\w$ on $\Zd$ as defined in \eqref{eq:def:nu} and assume that $\nu_l^{\w} (x) <\infty$ for all $x\in\Zd$.
  Then, for any $q> d/2$ there exist $\gamma\in (0,1]$, $\kappa\equiv \kappa (d,q)$, and $C_1\equiv C_1 (d,q)$ such that
  \begin{align}
    \max_{x \in Q_\eps} |u(x)|
    \;\leq\;
    C_1\,
    \Big(
      1 \vee \Norm{\nu_l^{\w}}{q, B_{1/\eps}}\, \Norm{f^\eps}{\ell^\infty (Q_\eps)} 
    \Big)^{\!\kappa}\, 
    \Norm{u}{\ell^2_\eps}^{\gamma}.
    \label{equ:linfty2D}
  \end{align}
\end{prop}
We prove this proposition after the following remark.
\begin{remark}
  \label{rem:linftyPoiss}
  Note that if $\ue:\, \Zdeps\to \R$ is a solution of \eqref{equ:epsPoisson}, then due to \eqref{equ:linfty1D}, \eqref{equ:PoincareBddSuppEps} and the Cauchy-Schwarz inequality it follows for any dimension $d\geq 1$ that
  \begin{align}
    \| \ue\|^2_{\ell^2_\eps} \;\leq\; C_{\rm S} \Norm{\nu_l^{\w}}{\frac{d}{2}, B_{1/\eps}}\,  \mathcal{E}_{\w}^\eps (\ue) \;\leq\; C_{\rm S} \Norm{\nu_l^{\w}}{\frac{d}{2}, B_{1/\eps}}\, \| \ue\|_{\ell^2_\eps}\,  \| f^\eps\|_{\ell^2_\eps (Q_\eps)}\, .
  \end{align}
  Let Assumption \ref{ass:Integ}\ref{ass:Integ:dhalf} be fulfilled.
  Then $\sup_{\eps>0}\| f^\eps\|_{\ell^2_\eps (Q_\eps)}<\infty$ implies by the ergodic theorem that both $\sup_{\eps>0}\| \ue\|_{\ell^2_\eps}$ and $\sup_{\eps>0}\mathcal{E}_{\w}^\eps (\ue)$ are bounded as well.
  Thus, \eqref{equ:linfty1D} implies that in dimension one $\sup_{\eps>0}\| \ue \|_\infty$ is bounded.
  Furthermore, if even Assumption \ref{ass:Integ}\ref{ass:Integ:q} is fulfilled and $\sup_{\eps>0}\| f^\eps\|_{\ell^\infty (Q_\eps)}<\infty$, then \eqref{equ:linfty2D} implies that $\sup_{\eps>0}\| \ue \|_\infty$ is bounded for $d\geq 2$ as well.
\end{remark}

\begin{proof}[Proof of Proposition \ref{prop:linftyPoiss}]
  We use the Moser iteration scheme.
  Let us fix $\eps>0$ and consider $u_\eps:\, \Zdeps\to \R$ with $\supp\, u_\eps\in Q_\eps$.
  We define $\tilde u^\alpha := |u|^\alpha \mathrm{sign}\, u$ for any $\alpha\geq 1$.
  By virtue of Eq.\ (A.2) in \cite{ADS15} we obtain the following energy estimate
  \begin{align}
    \mathcal{E}_\w^\eps (\tilde u_\eps^\alpha)
      &\leq \frac{\alpha^2}{2\alpha -1} \,\eps^d \sum_{x\in\Zd} \tilde u_\eps^{2\alpha-1}(\eps x) \left( -\Laeps u_\eps \right)(\eps x)\, .
      \label{equ:MoserHelp}
  \end{align}
  Since $\ue$ is a solution to the Poisson equation \eqref{equ:epsPoisson}, the energy estimate \eqref{equ:MoserHelp} implies that
  \begin{align*}
    \mathcal{E}_\w^\eps \left(\left( \tilde u^\eps\right)^\alpha\right)
    &\;\leq\; \frac{\alpha^2}{2\alpha -1} \| f^\eps \|_{\ell^\infty(Q_\eps)}\, \eps^d \sum_{x\in Q_\eps} \left( \tilde u^\eps(x)\right)^{2\alpha-1} 
      = \frac{\alpha^2}{2\alpha -1} \| f^\eps \|_{\ell^\infty(Q_\eps)} \left\| \ue \right\|_{\ell^{2\alpha -1}_\eps}^{2\alpha -1}
  \end{align*}
  By the Sobolev inequality \eqref{equ:SobolevEps} and Jensen's inequality it follows that
  \begin{align}
    \| \ue \|^{2\alpha}_{\ell^{2\alpha\rho}_\eps}
    \;\leq\; C_{\rm S}\frac{\alpha^2}{2\alpha -1} \| f^\eps \|_{\ell^\infty(Q_\eps)} \Norm{\nu_l^{\w}}{q, B_{1/\eps}} \left\| \ue\right\|_{\ell^{2\alpha}_\eps}^{2\alpha -1} \, ,
    \label{equ:MoserHelp3}
  \end{align}
  We define $\alpha_{j} = \rho^j$ for $j\in\mathbb{N}_0$.
  Further, we set $\gamma_j := 1-1/(2\alpha_j)$ for $\left\| \ue\right\|_{\ell^{2\alpha_j}_\eps}< 1$ and $\gamma_j := 1$ for $\left\| \ue\right\|_{\ell^{2\alpha_j}_\eps}\geq 1$.
  Recall that $\rho\equiv \rho(d,q)>1$ for any $q>d/2$.
  Furthermore, we observe that for any $\beta>0$ we have $\max_{x\in Q_\eps} |u (x)| \leq (2/\eps)^{d/\beta} \| u \|_{\ell^\beta_\eps}$.
  Thus, by iterating the inequality \eqref{equ:MoserHelp3} and using the fact that $\sum_{j=1}^\infty j/ \alpha_j<\infty$, we obtain that there exists $C_1 \equiv C_1 (d,q)<\infty$ such that
  \begin{align*}
    \| \ue \|_\infty
    \;\leq\; (2/\eps)^{d\eps} \| \ue \|_{\ell^{1/\eps}_\eps}
    \;\leq\; C_1 \left\| \ue \right\|_{\ell^2_\eps}^\gamma \prod_{j=0}^m \left( 1\vee \| f^\eps \|_{\ell^\infty(Q_\eps)}\, \| \nu^\w_l \|_{q,B_{1/\eps}}\right)^{\frac{1}{2\rho^{j-1}}}\,
  \end{align*}
  where $\gamma = \prod_{j=0}^m \gamma_j \leq 1$ and $m$ such that $2\alpha_m >1/\eps$.
  Choosing $\kappa = \sum_{j=0}^{\infty} 1/(2\alpha_j) <\infty$, we complete the proof.
\end{proof}

\section{Compact embedding}
\label{sec:Compactness}
The very first step to prove homogenization of the operator $\Laeps$ is to show that a sequence $\Repsadj \ue$ ($\ue \in \Heps$) has a strongly convergent subsequence if $\sup_\eps \mathcal{E}_\omega^\eps(u^\eps)<\infty$.
The Dirichlet energy $\mathcal{E}_\omega^\eps$ is defined in \eqref{equ:DefDirichEnEps}.

For any $m\mspace{-1mu}\in\mspace{-1mu}\mathbb{N}$ consider a partition of $Q$ into $m^d$ congruent open subcubes $(Q^m_j)_{j=1,\ldots, m^d}$ with side length $2/m$.
For a fixed $m$ we further define $Q^\eps_j := \supp\, \Repsadj\left(\Reps \mathds{1}_{Q_j^m}\right)$, where we suppress the superscript ``$m$'' although $Q^\eps_j$ depends on $m$.
Then $Q_j^m\subset Q^\eps_j$ and $|Q^\eps_j\backslash Q_j^m| \to 0$ as $\eps\to 0$.
  
\begin{lem}\label{lem:precompact}
  Let $\omega\in\Omega$ and assume that the uniform Poincar\'e inequality \eqref{equ:PI} holds with a finite $C(\omega)$ and that for any $m\in\N$ there exists $\eps^\ast_m>0$ such that for all $\eps<\eps^\ast_m$ we have
  \begin{align}\label{equ:Ass:Erg}
    \max_{1\leq j\leq m^d} \Norm{\nu_l^{\w}}{q, \eps^{-1}Q_j^\eps} \;\leq\; 2 \E\left[\left(\nu_l^\w(0)\right)^q\right]^{1/q}\, .
  \end{align}
  Then the Poincar\'e inequality \eqref{eq:Poincare:weighted} implies that for any sequence $\ue \in \Heps$ ($\eps^{-1}\in\mathbb{N}$) with $\sup_{\eps>0} \mathcal{E}^\eps_\w (\ue) <\infty$, the sequence $(\Repsadj \ue)_{\eps>0}$ has a strongly convergent subsequence in $L^2 (\R^d)$.
\end{lem}
This lemma was also recently shown in \cite[Lemma 3.14]{Neukamm2017}.

\begin{remark}
 If Assumptions \ref{ass:Environ} and \ref{ass:Integ}\ref{ass:Integ:dhalf} are fulfilled,
 then for \Pae realization $\omega\in\Omega$ the hypotheses of Lemma \ref{lem:precompact} are fulfilled.
 That is, by virtue of Assumptions \ref{ass:Environ}\ref{ass:erg}, \ref{ass:FullLattice} and \ref{ass:Integ}\ref{ass:Integ:dhalf} as well as Remark \ref{rem:PI}, there exists a \Pas finite $C(\omega)$ such that \eqref{equ:PI} is fulfilled.
 Furthermore, the same assumptions together with the ergodic theorem imply that \Pas there exists $\eps^\ast_m>0$ such that for all $\eps<\eps^\ast_m$ \eqref{equ:Ass:Erg} holds.
\end{remark}

\begin{proof}[Proof of Lemma \ref{lem:precompact}]
  First of all we observe that by virtue of \eqref{equ:PI} we have
  \begin{align*}
    \| \Repsadj \ue \|_2 \;=\; \| \ue \|_{\ell^2_\eps} \leq C(\omega) \mathcal{E}^\eps_\w (\ue)\, ,
  \end{align*}
  which implies that $\sup_{\eps>0} \| \Repsadj \ue \|_2$ is finite by assumption.
  By the Banach-Alaoglu theorem it follows that there exists a subsequence, which we still index by $\eps$, and $u\in \Hz$ such that
  \begin{align*}
    \Repsadj \ue \rightharpoonup u\quad\text{weakly in }L^2 (Q)\, .
  \end{align*}

  We now show that $u$ is also a strong limit.
  We estimate
  \begin{align}
    \| \Repsadj \ue - u \|_2^2
    &\leq 2\sum_{j=1}^{m^d}\mspace{10mu} \Biggl( \mspace{5mu}\| \Repsadj \ue - (\Repsadj\ue)_{Q_j^\eps} \|^2_{L^2 (Q_j^\eps)} +\nonumber\\
    &\qquad\qquad\quad+\| (\Repsadj\ue - u)_{Q_j^\eps} \|^2_{L^2 (Q_j^\eps)} + \| (u)_{Q_j^\eps} - u \|^2_{L^2 (Q_j^\eps)}\Biggr)\, ,
    \label{equ:CompactnessSplit}
  \end{align}
  where, in analogy to \eqref{equ:DefAverage}, we abbreviate 
  \begin{align*}
    (v)_{Q_j^\eps} \;\ldef\; |Q_j^\eps|^{-1} \int_{Q_j^\eps} v(x) \d x\quad\text{for $v\colon\, \Rd \to \R$.}
  \end{align*}
  Since $\Repsadj \ue$ converges weakly in $L^2(Q)$ to $u$, the sum over the second term on the RHS of \eqref{equ:CompactnessSplit} vanishes as $\eps$ tends to zero.
  It remains to show that, as the $\eps\to0$, the limit superior of the sum of the first and third term is zero as well.
  
  We use arguments similar to the ones given in \cite[Proposition 2.9]{ADS15}, see also \cite[Lemma 3.14]{Neukamm2017}.
  Let $\vec{e}_i$ ($i=1,\ldots,d$) be the unit base vectors of $\Rd$.
  By virtue of Proposition \ref{prop:localPoincare} there exists $C_{\rm PI}<\infty$ such that \Pas for $\eps$ small enough the first term in the brackets of the RHS in \eqref{equ:CompactnessSplit} can be estimated by
  \begin{align}
    \| \Repsadj \ue - (\Repsadj\ue)_{Q_j^\eps} \|^2_{L^2 (Q_j^\eps)}
    &\;=\; \| \ue - (\ue)_{Q_j^\eps} \|^2_{\ell^2_\eps (Q_j^\eps)}\nonumber\\
    &\;\leq\; C_{\rm PI} \Norm{\nu_l^{\w}}{q, \eps^{-1}Q_j^\eps}\,\frac{4\eps^d}{m^2}\,   \sum_{i=1}^{d}\sum_{x,x+\vec{e}_i\in C_{\rm W}\eps^{-1}Q_j^\eps} \w_{x,\vec{e}_i} \left( \del_{\vec{e}_i}^\eps \ue (\eps x) \right)^2
    \label{equ:equ:CompactnessSplit2}
  \end{align}
  where for $d=1$ we set $l=C_{\rm W}=q= 1$. For $d\geq 2$ we set $q=d/2$.
  Since any edge $e\in E_d$ is contained in at most $C_{\rm o}\ldef 2dC_{\rm W}$ cubes $C_{\rm W}\eps^{-1}Q_j^\eps$, summing over $j=1,\ldots,m^d$ yields
  \begin{align}
    2\sum_{j=1}^{m^d} \| \Repsadj \ue - (\Repsadj\ue)_{Q_j^\eps} \|^2_{L^2 (Q_j^\eps)}
    \;\leq\; 8m^{-2} C_{\rm PI} C_{\rm o}\, \mathcal{E}_\w^\eps (\ue) \max_{1\leq j\leq m^d} \Norm{\nu_l^{\w}}{q, \eps^{-1}Q_j^\eps}\, .
    \label{equ:equ:CompactnessSplit3}
  \end{align}
  Note that $C_{\rm o}$ is independent of $m$ and $\mathcal{E}_\w^\eps (\ue)$ is bounded in $\eps$ by assumption.
  
  By virtue of \eqref{equ:Ass:Erg}, \eqref{equ:CompactnessSplit} and \eqref{equ:equ:CompactnessSplit3} it follows that there exists $C<\infty$ independent of $m$ such that $\Pr$-almost surely
  \begin{align*}
    \limsup_{\eps\to 0}\| \Repsadj \ue - u \|_2^2
    &\;\leq\; Cm^{-2} + 2\sum_{j=1}^{m^d} \limsup_{\eps\to 0} \| (u)_{Q_j^\eps} - u \|^2_{L^2 (Q_j^\eps)}\\
    &\;=\; Cm^{-2} + 2\| u- \mathcal{R}_{2/m}^\ast \mathcal{R}_{2/m} u\|^2_2\, .
  \end{align*}
  Since $m$ might be arbitrarily large and $u\in L^2(Q)$ has bounded support, the claim follows.
\end{proof}

\section{Analytic tools}
\label{sec:AnalyticTools}
In this section we always assume that the law $\Pr$ is stationary and ergodic with respect to spatial translations.

\subsection{An ergodic theorem}
In what follows, we will generalize a result by Boivin and Depauw.
\begin{thm}[{Ergodic Theorem by Boivin and Depauw \cite[Theorem 3]{BoivinDepauw2003}}]
\label{thm:Boiv-Dep}
For every $f\in L^{1}(\Omega,\P)$, for $\P$-almost every $\omega\in\Omega$ it holds
\begin{align}
\lim_{\eps\to0}\mspace{7mu}\eps^d \sum_{x\in \eps^{-1} Q_\eps} v(\eps x)f(\tau_x\omega)
\;=\;\mathbb{E}[f]\int_{Q}v(x)\d x\qquad\forall v\in C(\overline{Q})\,,
\label{equ:Boiv-Dep}
\end{align}
and the Null-set depends on $f$ but not on $v$.\end{thm}
\begin{remark}
\label{rem:Tempelman}Evidently, we can also choose $v$ as the
characteristic function of any relatively open or compact set $A\subset Q$ and we
obtain the Tempel'man ergodic theorem.
\end{remark}
We will use both Theorem \ref{thm:Boiv-Dep} and Remark \ref{rem:Tempelman} in order to prove the following theorem.
\begin{thm}
\label{thm:General-ergodic-thm}\
For every $f\in L^{1}(\Omega,\P)$, for $\P$-almost every $\omega\in\Omega$ the following holds:
Let $(\ue)_{\eps>0}$ be a sequence of functions from $\eps \Z^d \to \R$ with support in $Q_\eps$ such that $\Repsadj \ue\to u$ pointwise a.e.\ in $Q$.
Furthermore, let $\sup_{\eps>0}\left\Vert \ue\right\Vert _{\infty}<\infty$.
Then $u\in L^\infty (Q)$ and
\begin{equation}
\lim_{\eps\to0} \mspace{7mu}\eps^d \sum_{x\in \eps^{-1} Q_\eps} \ue(\eps x) f\left(\tau_x \omega\right)
\;=\;\mathbb{E}[f]\int_{Q}u(x)\d x
\label{eq:thm:General-ergodic-thm}
\end{equation}
and the Null-set depends on $f$ but not on the sequence $\ue$.\end{thm}
\begin{proof}
First we note that $u\in L^\infty (Q)$ since $\sup_{\eps>0}\left\Vert \ue\right\Vert _{\infty}<\infty$.
Now we let $\eta>0$ and let $\rho_\delta$ be a sequence of mollifiers approximating the identity. By Egorov's theorem, there exists a compact set $K_{\eta}$
with ${\Leb}\left(Q\backslash K_{\eta}\right)<\eta$ such that both $\Repsadj\ue\to u$
and $u_{\delta}:=u\ast\rho_{\delta}\to u$ uniformly on $K_{\eta}$.
We now make the following decomposition:
\begin{align}
  &\left| \mspace{7mu}\eps^d \sum_{x\in \eps^{-1} Q_\eps} \ue(\eps x) f\left(\tau_x \omega\right) - \mathbb{E}[f]\int_{Q}u(x)\d x\mspace{7mu}\right|\nonumber\\
  &\qquad\quad\leq \left| \mspace{7mu}\eps^d \sum_{x\in \eps^{-1} Q_\eps} \left( \ue(\eps x) - u_\delta (\eps x)\right) f\left(\tau_x \omega\right) \mspace{7mu}\right| +\nonumber\\
  &\qquad\quad\phantom{\leq}+\left| \mspace{7mu}\eps^d \sum_{x\in \eps^{-1} Q_\eps} u_\delta (\eps x) f\left(\tau_x \omega\right) - \mathbb{E}[f]\int_{Q}u_\delta (x)\d x\mspace{7mu}\right| +
  \left| \mspace{7mu}\mathbb{E}[f]\int_{Q}\left( u_\delta (x) - u(x) \right) \d x\mspace{7mu}\right|
  \label{equ:PfGenErgodichelp2}
\end{align}
Since $u_\delta\in C(\overline{Q})$, the second summand on the above RHS converges to zero by virtue of Theorem \ref{thm:Boiv-Dep}.
For the first summand on the RHS of \eqref{equ:PfGenErgodichelp2} we estimate that
\begin{align}
\lim_{\eps\to0}\;&\left| \eps^d \sum_{x\in \eps^{-1} Q_\eps} \left( \ue(\eps x) - u_\delta(\eps x)\right) f\left(\tau_x \omega\right) \right|\nonumber\\
&\qquad\qquad\qquad\qquad\qquad \leq\lim_{\eps\to0}\sup_{x\in K_{\eta}}\left|\ue(x)-u_\delta (x)\right|\, \eps^d \sum_{x\in \eps^{-1}(K_{\eta}\cap Q_\eps)}\left|f\left(\tau_x\omega\right)\right|\nonumber\\
&\qquad\qquad\qquad\qquad\qquad\qquad+\lim_{\eps\to0}\left(\left\Vert u_\delta\right\Vert _{\infty}+\left\Vert \ue\right\Vert _{\infty}\right)\, \eps^d \sum_{x\in \eps^{-1}Q_\eps \backslash K_{\eta}} \left|f\left(\tau_x\omega\right)\right|\, .
 \label{equ:PfGenErgodichelp1}
\end{align}
Since the function $\Repsadj u^\eps$ converges uniformly in $\eps$ to $u$ on $K_{\eta}$,
we can estimate by virtue of Remark \ref{rem:Tempelman} that
\begin{align*}
\lim_{\eps\to0}\;\sup_{x\in K_{\eta}}\left|\ue(x)-u_\delta (x)\right|\; \eps^d\mspace{-9mu} \sum_{x\in \eps^{-1}(K_{\eta}\cap Q_\eps)}\left|f\left(\tau_x\omega\right)\right|
\;\leq\; \sup_{x\in K_{\eta}}\left|u_{\delta}(x)-u(x)\right|\,|Q|\;\E [f]\, .
\end{align*}
We further estimate the second summand on the RHS of \eqref{equ:PfGenErgodichelp1} by
\begin{align*}
  \lim_{\eps\to0}\;\left(\left\Vert u_\delta\right\Vert _{\infty}+\left\Vert \ue\right\Vert _{\infty}\right)\; \eps^d\mspace{-9mu} \sum_{x\in \eps^{-1}Q_\eps \backslash K_{\eta}} \left|f\left(\tau_x\omega\right)\right|
 &\;\leq\;2\eta \sup_{\eps>0}\left\Vert \ue\right\Vert _{\infty} \E [f]\,,
\end{align*}
where we have used  Remark \ref{rem:Tempelman}.

Thus, as $\eps\to0$, we obtain that
\begin{align*}
  &\lim_{\eps\to 0}\;\left| \eps^d \sum_{x\in \eps^{-1} Q_\eps} \ue(\eps x) f\left(\tau_x \omega\right) - \mathbb{E}[f]\int_{Q}u(x)\d x\right|\\
  &\qquad\quad\leq \sup_{x\in K_{\eta}}\left|u_{\delta}(x)-u(x)\right|\,|Q|\,\E [f] \;+\; 2\eta \sup_{\eps>0}\left\Vert \ue\right\Vert _{\infty} \E [f] \;+\; 
  \left| \mathbb{E}[f]\int_{Q}\left( u_\delta (x) - u(x) \right) \d x\right|
\end{align*}
As $\delta\to0$, the uniform convergence $u_\delta \to u$ on $K_\eta$ yields
\begin{align*}
\lim_{\eps\to 0}\;\left|\; \eps^d\mspace{-9mu} \sum_{x\in \eps^{-1} Q_\eps} \ue(\eps x) f\left(\tau_x \omega\right) \;-\; \mathbb{E}[f]\int_{Q}u(x)\d x\;\right|
\;\leq\; 2\eta\, \sup_{\eps>0}\left\Vert \ue\right\Vert _{\infty}\E [f]\,.
\end{align*}
Since the last inequality holds for every $\eta>0$, the claim follows.
\end{proof}

\subsection{Function spaces}
In what follows, we always assume that Assumption \ref{ass:Environ}\ref{ass:momcond} holds.
We first note that the probability space given in \eqref{eq:def-prob-space}
is generated from the compact metric space $[0,\infty]^{E}$, and
therefore the notion of continuity on $\Omega$ makes sense.
We say that a function $\varphi:\, \Omega\times \Zd \to \R$ is \emph{shift covariant} if it fulfills
\begin{align}
  \varphi(\omega,x+z)-\varphi(\omega,x)\;=\;\varphi(\tau_{x}\omega,z)
  \label{equ:shiftcov}
\end{align}
for all $x,z\in\Zd$ (cf. \cite{Biskup2011review} Eq. (3.14)).
Note that shift covariant functions $\varphi$ fulfill $\varphi(\w,0)=0$.
Then \eqref{equ:shiftcov} directly implies that
\begin{align}
  \varphi(\omega,x)=-\varphi(\tau_{x}\omega,-x)\, .
  \label{equ:shiftcov2}
\end{align}
We define on $\Omega\times\Zd$ the space 
\begin{align*}
L_{\rm cov}^{2} & :=\left\{ \varphi:\,\Omega\times\Zd\to\R\,\,:\,\,\varphi\text{ satisfies \eqref{equ:shiftcov} and }\left\Vert \varphi\right\Vert _{L_{\rm cov}^{2}}<\infty\right\}\, ,\\
\text{where }
\left\Vert \varphi\right\Vert^2_{L_{\rm cov}^{2}} &:=\E\left[ \sum_{z\in\Zd}\omega_{0,z} \varphi(\omega,z)^{2}\right]\,.
\end{align*}
Accordingly, we define the scalar product between $\varphi_1, \varphi_2\in L^2_{\rm cov}$ by
\begin{align}
  \langle \varphi_1, \varphi_2 \rangle_{L^2_{\rm cov}} := \E\left[ \sum_{z\in\Zd}\omega_{0,z} \varphi_1(\omega,z) \varphi_2(\omega,z)\right]\, .\label{equ:scalarprodCov}
\end{align}

Note that $L_{\rm cov}^{2}$ is a closed subspace of $\bigotimes_{z\in\Zd} L^2(\Omega,\mu_z)$, where $\mu_z$ is the measure on $\Omega$ defined by $\d\mu_z(\omega)=\omega_{0,z}\d\Pr(\omega)$.
Since $\Omega$ is a compact metric space, $L^2(\Omega,\mu_z)$ is separable for all $z\in\Zd$ and thus also the countable product space $\bigotimes_{z\in\Zd} L^2(\Omega,\mu_z)$ and its subspace $L_{\rm cov}^{2}$ are separable. 

Further, we note that for all $\phi:\Omega\to\R$ it holds that $\mD\phi(\omega,z):=\mD_{z}\phi(\omega):=\phi(\tau_{z}\omega)-\phi(\omega)$ satisfies $\mD\phi(\omega,x+z)-\mD\phi(\omega,x)=\mD\phi(\tau_{x}\omega,z)$.
Therefore $\mD\phi$ is in $L^2_{\rm cov}$.
A \emph{local function} on $\Omega$ is a bounded, continuous function that only depends on finitely many coordinates of $[0,\infty]^{E}$.
Following the outline of Chapter 3 in \cite{Biskup2011review}, we define the closed subspace 
\[
L_{\rm pot}^{2}\;\ldef\;\overline{\left\{ \mD\phi\,:\,\phi\mbox{ local}\right\} }^{L_{\rm cov}^{2}}\,.
\]
Let $L_{\rm sol}^{2}$ be the orthogonal complement of $L_{\rm pot}^{2}$ in $L^2_{\rm cov}$ and let us define
\begin{align*}
  \mathrm{div}\left(\omega b\right)\;\ldef\;\sum_{z}\omega_{0,z}\left(b(\omega,z)-b(\tau_{z}\omega,-z)\right)\, .
\end{align*}
Note that since $b$ satisfies \eqref{equ:shiftcov2}, the last equation also reads 
\begin{equation}\label{eq:divergence}
\mathrm{div}\left(\omega b\right) \;=\; 2 \sum_{z}\omega_{0,z} b(\omega,z).
\end{equation}
Then we have the following lemma.
\begin{lem}[{\cite[Lemma 3.6]{Biskup2011review}}]
\label{lem:DivSol}
  \begin{align}
    \mathrm{div}\left(\omega b\right) \;=\; 0 \qquad\text{for all }b\in L_{\rm sol}^{2}\text{ and $\Pr$-a.a. }\w\, .
  \end{align}
\end{lem}
Using the above notation, we define $\chi\in\left(L_{\rm pot}^{2}\right)^{d}$
through
\begin{equation}
\chi=\mathrm{argmin}\left\{ \E\left[\sum_{z\in\Zd}\w_{0,z}|z+\tilde{\chi}(\w,z)|^2\right]\,:\,\tilde{\chi}\in\left(L_{\rm pot}^2\right)^{d}\right\} \,,\label{eq:definition-corrector}
\end{equation}	
i.e., $\chi_j$ is the orthogonal projection of $z_j\in L^2_{\rm cov}$ on the space $L^2_{\rm pot}$ with respect to the scalar product defined in \eqref{equ:scalarprodCov}.
We will see below that we can write the homogenized matrix as
\begin{equation}
\left(A_{\rm hom}\right)_{i,j}=\E\left[ \sum_{z\in\Zd}\w_{0,z}\left(\vec{e}_{i}\cdot\left[z+\chi(\w,z)\right]\right)\left(\vec{e}_{j}
\cdot\left[z+\chi(\w,z)\right]\right)
\right]\,,\label{eq:Definition-A-hom}
\end{equation}
where the $\vec{e}_i$, $i=1,\ldots,d$, denote the unit base vectors of $\Rd$.
In analogy to \cite[Lemma 4.5]{Faggionato2008} we know the following result.
\begin{lem}
\label{lem:non-degeneracy}
Suppose that $\E\big[\nu_l^{\w}(0)\big] < \infty$ with $\nu_l^{\w}$ as defined in \eqref{eq:def:nu}.
Then the matrix $A_{\mathrm{hom}}$ is positive definite.
In particular, the vectorial space spanned by the following vectors
\begin{align}
  \E\Big[
   {\textstyle \sum_{z \in \Zd}}\, \w(0, z)\, z b(\w, z)
  \Big] \in \Rd,
  \qquad b \in L^2_{\mathrm{sol}}
  \label{eq:Non-degeneracy-consequence}
\end{align}
coincides with $\Rd$.
\end{lem}

\begin{proof}
    First we notice that $\psi(\cdot, \vec{e}_i) \in L^1(\Omega, \P)$ for any $\psi \in L_{\mathrm{cov}}^2$ and $i=1,\ldots,d$, provided that $\E[\nu_l^{\w}(0)] < \infty$.
    Indeed, by the Cauchy-Schwarz inequality and the shift covariance \eqref{equ:shiftcov}, we observe that
  \begin{align}\label{eq:L1:est}
    \E\big[|\psi(\w, \vec{e}_i)|\big]
    &\;\leq\;
    \E\big[1/\w_l(0, \vec{e}_i)\big]^{1/2}\,
    \biggl(\E\biggl[\sum_{x,y\in\gamma_l^{\rm opt}} \w (\{x,y\}) |\psi (\tau_x \w, x-y)|^2\biggr]\biggr)^{1/2}
    \nonumber\\
    &\;\leq\;
    \sqrt{l\;|\Gamma_l|}\, \E\big[\nu_l^{\w}(0)\big]^{1/2}\,
    \Norm{\psi}{L_\mathrm{cov}^2}\, ,
  \end{align}
  where we abbreviate $\gamma_l^{\rm opt} = \gamma_l^{\rm opt} (\{ 0,\vec{e}_i \})$, recall \eqref{equ:Defwl}.
  Moreover, by adapting the argument given in \cite[Proof of Lemma~4.8]{Biskup2011review}, it follows that $\E[\psi(\w, \vec{e}_i)] = 0$ for any $\psi \in L_{\mathrm{pot}}^2$ and $i=1,\ldots,d$.
  In particular, $\E[\chi_j(\w,\vec{e}_i)] = 0$ for any $i,j = 1, \ldots, d$. 

  Now let $v \in \R^d \backslash \{0\}$. 
  Since $\E[v \cdot \chi(\w, \vec{e}_i)] = 0$, it follows that
  \begin{align*}
    (v \cdot \vec{e}_i)^2
    &\overset{\phantom{\eqref{eq:L1:est}}}{\;=\;}
    (v\cdot \vec{e}_i)\, \E\Big[\bigl(v \cdot [\vec{e}_i + \chi (\w,\vec{e}_i)]\bigr)\Big]
     \\[.5ex]
     &\overset{\eqref{eq:L1:est}}{\;\leq\;}
     |v \cdot \vec{e}_i|\, \sqrt{l\;|\Gamma_l|}\, \E\big[\nu_l^{\w}(0)\big]^{1/2}\,
     \E\,\biggl[ \sum_{z\in\Zd} \w_{0,z} \bigl( v \cdot [z + \chi(\w,z)]^2\bigr) \biggr]\,.
  \end{align*}
  Thus, by summing both sides over $i=1,\ldots,d$, we obtain
  \begin{align*}
    \sqrt{(v, A_{\mathrm{hom}} v)}
    \;=\;
    \E\,\biggl[ \sum_{z\in\Zd} \w_{0,z} \bigl( v \cdot [z + \chi(\w,z)]^2\bigr) \biggr]
    \;\geq\;
    \frac{|v|_2^2}{|v|_1}\, \,\big(l\,|\Gamma_l|\, \E\big[\nu_l^{\w}(0)\big]\big)^{-1/2}
    \;>\;
    0\, .
  \end{align*}
  Thus, the matrix $A_{\mathrm{hom}}$ is positive definite.  By following literally the proof of \cite[Lemma~4.5]{Faggionato2008} we obtain the claim.
\end{proof}

\paragraph{Bochner spaces}
We will use the concept of Bochner spaces, which are a special case of the theory outlined in \cite{Ma2002BanachMeasures}. Let $X$ be a normed space with norm $\left\Vert \cdot\right\Vert _{X}$
with the corresponding topology and Borel-$\sigma$-algebra and let
$U\subset\Rd$ be a Lebesgue-measurable set. Then, for $1\leq p<\infty$,
we define the space 
\begin{align*}
\left\Vert f\right\Vert _{L^{p}(U;X)} & :=\left(\int_{U}\left\Vert f(x)\right\Vert _{X}^{p}dx\right)^{\frac{1}{p}}\,,\\
L^{p}(U;X) & :=\left\{ f:\,U\to X\,:\,f\mbox{ is measurable and }\int_{U}\left\Vert f(x)\right\Vert _{X}^{p}\d x<\infty\right\} \,.
\end{align*}
Given a measure space $(\Omega,{\mathcal{F}},\P)$, it turns out that $L^{p}(U;L^{p}(\Omega,\P))$
and $L^{p}(U\times\Omega;\mathscr{L}\otimes\P)$ are isometrically
isomorph via the trivial identification $f(x)(\omega)=f(x,\omega)$.
Here, $\mathscr{L}$ denotes the Lebesgue measure and $\mathscr{L}\otimes\P$
denotes the product measure. While not being necessary, this notation
has proved useful in homogenization theory since the introduction
of two-scale convergence in \cite{Allaire1992}.
In particular, it gives a clear and intuitive meaning to spaces such as 
\begin{align*}
L^{2}(Q;L_{\rm cov}^{2}) & :=\left\{ \varphi:\,Q\times\Omega\times\Zd\to\R\,:\,\int_{Q} \left\| \varphi(x,\cdot, \cdot)\right\|_{L^2_{\rm cov}} \d x<\infty\,,\right.\\
 & \qquad\qquad\qquad\qquad\qquad\qquad\left.\phantom{\int_{Q}\sum_{z\in\Zd}}\,\varphi(x,\cdot,\cdot)\in L_{\rm cov}^{2}\,\mbox{ for a.e. }x\in Q\right\} 
\end{align*}
or $L^{2}(Q;L_{\rm pot}^{2})$.

If $\tilde{X}\subset X$ is a family of vectors in $X$, we denote
\[
C(\overline{Q})\otimes\tilde{X}:={\rm span}\left\{ xf\,\,:\,\,f\in C(Q)\,,\,\,x\in\tilde{X}\right\} \,.
\]
If $\tilde{X}$ is a countable dense subset of $X$, i.e. $X$ is
separable, every element of $L^{2}(Q;X)$ can be approximated by finite
sums of elements of $C(\overline{Q})\otimes\tilde{X}$ \cite{Ma2002BanachMeasures}. 

\subsection{Discrete derivatives}
\label{sec:discrder}
With the following definitions of discrete derivatives, we can write the operator $\Laeps$ in divergence form.
\begin{defn}[Discrete derivatives]
	For $u\colon\,\Zdeps\to\R$ we define the $\eps$-forward derivative in the direction $z\in \Z^d$ by
	\begin{align}
		\del^\eps_z u (x) = \eps^{-1} \left( u(x+\eps z) - u(x) \right)\, ,
	\end{align}
	and the analogous backward derivative,
	\begin{align}
		\del^{\eps -}_z u (x) = \eps^{-1} \left( u(x) - u(x-\eps z) \right)\, .
	\end{align}
	Further, we define $\nabla^\eps u(x,z):=\del^\eps_z u (x)$ and write $\nabla^\eps u(x)$ for the function that maps $z\in\Zd$ to $\nabla^\eps u(x,z)$.
	Accordingly, we define $\nabla^{\eps-}u(x,z):=\del^{\eps -}_z u (x)$ and $\nabla^{\eps-}u(x)$.
	Moreover, for a function $v:\,\Zdeps\times\Zd\to\R$ we define
	\begin{align}
	  {\rm div}^\eps v (x) = \sum_{z\in\Zd} \del^{\eps -}_z v (x,z)\, .
	\end{align}
\end{defn}
We use this notation to clearly distinguish between $\nabla^\eps$, an operator on discrete functions, and $\nabla$, an operator on the Sobolev space $H^1 \left( \R^d \right)$. A direct calculation shows that when $A_\omega^\eps$ maps $v(x,z)\mapsto \omega_{\frac x\eps,z}v(x,z)$, then
\begin{align}
 -\Laeps u^\eps=-\frac{1}{2}\,\mathrm{div}^\eps\left(A_\omega^\eps \nabla^\eps u^\eps\right)\,.
 \label{equ:DiscreteDivForm}
\end{align}
Moreover, for $v^\eps\colon\, \Zdeps\to\R$ we observe that
\begin{align}
  \langle -\Laeps u^\eps, v^\eps\rangle_\Heps = \frac{\eps^d}{2} \sum_{x\in\Zd}\sum_{z\in\Zd} \w_{x,z} \bigl(\del_z^\eps \ue (\eps x)\bigr) \bigl(\del_z^\eps v^\eps (\eps x)\bigr)\, .
  \label{eq:Laeps-strictly-positive-def}
\end{align}
When we compare the divergence form of the operator $\Laeps$ in \eqref{equ:DiscreteDivForm} with the limit operator in \eqref{eq:strong-limit-equation}, we better understand the result of Theorem \ref{thm:twoscale-strong}.
Furthermore \eqref{eq:Laeps-strictly-positive-def} implies that $\Laeps$ is strictly positive definite on any bounded domain with zero Dirichlet conditions at the boundary.

\subsection{Two-scale convergence}
\label{sec:two-scale}

We adapt the concept of stochastic two-scale convergence by Zhikov
and Piatnitsky \cite{zhikov2006homogenization} to our setting. 

We denote by $z_{i}$ the function that maps $z\in\Zd$ onto its $i$'th
coordinate and observe that, since $\E\left[\sum_{z\in\Zd} \w_{0,z}|z|^2 \right]$ is finite, $z_{i}\in L_{\rm cov}^{2}$ for $i=1,\dots,d$.

Since $L_{\rm cov}^{2}$ is separable, there exist countable sets $\Phi_{\rm sol}\subset L_{\rm sol}^{2}$
and $\Phi_{\rm pot}\subset L_{\rm pot}^{2}$ such that $\Phi:=\Phi_{\rm sol}\oplus\Phi_{\rm pot}\oplus\{z_{1},\dots,z_{d}\}\oplus\{1\}$
is dense in $L_{\rm cov}^{2}$. We can assume that every $\varphi\in\Phi_{\rm pot}$
is the gradient of a local function. Furthermore, there exists a countable
subspace $\Psi\subset C_{\rm c}^\infty (\Rd)$ such that $\Psi$ is dense both in $L^{2}(\Rd)$ and in $C_{\rm c}^\infty (\Rd)$.
We then find that $\Psi\otimes\Phi$ is dense in $L^{2}(\Rd;L_{\rm cov}^{2})$.

\begin{defn}[Typical realizations]
We denote by $\Omega_{\Phi}\subset\Omega$ the set of all $\omega\in\Omega$
such that Theorem \ref{thm:Boiv-Dep} holds
\begin{enumerate}
 \item[a)] for all $f(\omega):=\sum_{z\in\Zd}\omega_{0,z}\varphi(\omega,z)$, where $\varphi\in\Phi$,
 \item[b)] for all $f(\omega):=\sum_{z\in\Zd}\omega_{0,z}\left(\varphi_{i}\varphi_{j}\right)(\omega,z)$, where $\varphi_{i},\varphi_{j}\in\Phi$, and
 \item[c)] and for all $f(\omega):=\sum_{z\in\Zd\backslash Z}\omega_{0,z}|z|^{2}$,
  where $Z$ is a finite subset of $\Zd$,
 \item[d)] ${\rm div}(\omega b)\circ \tau_x = 2 \sum_{z}\omega_{x,z} b(\tau_x \omega,z)= 0$ for all $b\in\Phi_{sol}$ and all $x\in\Zd$.
\end{enumerate}
We call $\Omega_{\Phi}$ the set of \emph{typical realizations}.
\end{defn}

\begin{remark}
  Note that $\P (\Omega_{\Phi}) = 1$ (compare to \cite[Lemma 4.4]{Faggionato2008}).
\end{remark}

\begin{defn}[Two-scale convergence]
Let $\rmw_{\eps}:\,\eps\Z^d\times\Zd\to\R$.
We say that $\rmw_{\eps}$ \emph{converges weakly in two scales} to $\rmw\in L^{2}(\Rd;L_{\rm cov}^{2})$ if
\begin{align}
  \lim_{\eps\to 0}\eps^d \sum_{x\in \Zd} v(\eps x) \sum_{z\in\Zd} \omega_{x,z} \rmw_\eps (\eps x,z) \varphi(\tau_x\w, z) = \int_{\Rd} v(x) \E \left[ \sum_{z\in\Z^d}\w_{0,z} \rmw(x,\w,z) \varphi(\w,z) \right]\d x
\end{align}
for all $v\in C_{\rm c}^{\infty}(\Rd)$ and all $\varphi\in\Phi$.
In this case we write $\rmw_{\eps}\stackrel{2s}{\weakto}\rmw$.
\end{defn}

\begin{prop}
\label{prop:Ex-two-scale-limit-farfield}For all typical realizations
$\omega\in\Omega_{\Phi}$ it holds: If $\rmw_{\eps}:\,\eps\Z^d\times\Zd\to\R$ and $C<\infty$ are such that 
\begin{align}
  \eps^d \sum_{x\in\Zd} \sum_{z\in\Zd} \w_{x,z} \rmw^2_\eps (\eps x,z)\;\leq\; C\qquad\forall\eps>0\,,
  \label{equ:BddTwoScale}
\end{align}
then there exists a subsequence $\rmw_{\eps_{k}}$ and $\rmw\in L^{2}(\Rd;L_{\rm cov}^{2})$
such that 
\begin{align}
\rmw_{\eps_k}\stackrel{2s}{\weakto}\rmw\, .
\label{eq:two-sc-conv-1}
\end{align}
\end{prop}

\begin{proof}
The proof goes along the lines of classical proofs of two-scale convergence
like for example in \cite{zhikov2006homogenization}, Section 5.

We observe that for every $v\in\Psi$ and $\varphi\in\Phi$ we find
\begin{align}
\limsup_{\eps\to 0}\eps^d &\left| \sum_{x\in \Zd} v(\eps x) \sum_{z\in\Zd} \omega_{x,z} \rmw_{\eps} (\eps x,z) \varphi(\tau_x\w, z) \right|\nonumber\\
&\stackrel{\eqref{equ:BddTwoScale}}{\leq} \limsup_{\eps\to 0}\sqrt{C} \left( \eps^d \sum_{x\in\Zd}\sum_{z\in\Zd} \w_{x,z} v^2(\eps x) \varphi^2(\tau_x\w,z) \right)^{1/2}
\stackrel{\eqref{equ:Boiv-Dep}}{\leq} \sqrt{C} \| v \|_{L^2} \| \varphi \|_{L^2_{\rm cov}}\, ,
\label{eq:two-scale-exist-estim-1}
\end{align}
where, in the last step, we have also used that $v$ has bounded support.
It follows that since $\Psi$ and $\Phi$ are countable, we can choose a subsequence
$\eps_{k}\to0$ as $k\to\infty$ such that the limit $I(v\varphi)$ of
\begin{align*}
  I_{\eps_k} (v\varphi)\;\ldef\; \eps_k^d \sum_{x\in \Zd} v(\eps_k x) \sum_{z\in\Zd} \omega_{x,z} \rmw_{\eps_k} (\eps_k x,z) \varphi(\tau_x\w, z) 
\end{align*}
exists for every $v\in\Psi$ and $\varphi\in\Phi$.
We notice that the functional $I(\cdot)$ is linear in $v\varphi\in\Psi\otimes\Phi$.
Furthermore, due to (\ref{eq:two-scale-exist-estim-1}), $I(\cdot)$ is continuous on ${\rm span}\left\{ \Psi\otimes\Phi\right\} $. It follows by Riesz representation theorem that we can find $\rmw\in L^{2}(Q;L_{\rm cov}^{2})$ such that
\[
I(v\varphi)\;=\;\int_{\Rd} v(x) \E \left[ \sum_{z\in\Z^d}\w_{0,z} \rmw(x,\w,z) \varphi(\w,z) \right]\, \d x\,.
\]
Since $\Psi\otimes\Phi$ is dense in $L^{2}(Q;L_{\rm cov}^{2})$, we obtain
that $\rmw$ is uniquely defined.
Since, in addition, $\Psi$ is dense in $C_{\rm c}^\infty$ we find for every $v\in C_c^{\infty}(\Rd)$ and $\varphi\in\Phi$  $I_{\eps_k} (v\varphi)\to I(v\varphi)$ as $\eps\to0$ and hence \eqref{eq:two-sc-conv-1} holds.
\end{proof}

\begin{lem}
\label{lem:BddDiscrGrad}
  For all typical realizations $\w\in\Omega_{\Phi}$ and all Lipschitz functions $v:\, \Rd\to\R$ there exists $C(\omega)\in (0,\infty]$, which depends only on $\supp v$ and $\w$, such that
  \begin{align}
      \sup_{\eps>0} \eps^d \sum_{x\in\Zd}\sum_{z\in \Zd}\omega_{x,z}\left(\del_z^\eps v(\eps x)\right)^{2}
      \;<\; C(\omega) \| \nabla v \|^2_\infty\,.
  \end{align}
  If $\supp v$ is bounded, then $C(\omega)$ is \Pas finite.
\end{lem}
\begin{proof}
We observe that we can interchange the order of the sums and estimate
\begin{align*}
  \eps^d \sum_{z\in \Zd}\sum_{x\in\Zd}&\omega_{x,z}\left(\del_z^\eps v(\eps x)\right)^{2}
  \;\leq\; \eps^d\, \| \nabla v \|^2_\infty\sum_{z\in \Zd}\sum_{x\in\eps^{-1} (\supp v \cup (\supp v - \eps z))} \w_{x,z} |z|^2\\
  &\;\leq\; \eps^d\, \| \nabla v \|^2_\infty\sum_{x\in\eps^{-1} \supp v} \sum_{z\in \Zd} \mspace{-6mu}\w_{x,z} |z|^2 \;+\; \eps^d\, \| \nabla v \|^2_\infty\sum_{z\in \Zd}\sum_{x\in(\eps^{-1}\supp v) - z} \w_{x,z} |z|^2
\end{align*}
The first term on the above RHS is finite by virtue of the ergodic theorem. This also holds for the second term after an index shift in $x$ and a rearrangement of the two sums.
\end{proof}

\begin{lem}
\label{lem:TestWithGradient}
Let $Q_\eps = Q\cap \eps\Z^d$.
For all typical realizations $\omega\in\Omega_{\Phi}$ it holds:
\begin{align}
  \limsup_{\eps\to 0}\, \eps^d \sum_{x\in \Z^d}\sum_{z\in \Zd}\omega_{x,z}\left(\del^\eps_z v(\eps x)-\nabla v(\eps x)\cdot z\right)^{2} = 0
  \quad\text{for all }v\in C_{\rm c}^\infty (\Rd)\, .
  \label{equ:TestWithGradient}
\end{align}
\end{lem}
\begin{proof}
Let $\delta>0$.
Since $\E\left[ \sum_{z\in\Zd}\omega_{0,z}\left|z\right|^{2}\right]<\infty$, we can choose a finite point-symmetric subset $Z_{\delta}\subset\Zd$ such that
\begin{align*}
  \E\left[ \sum_{z\in\Zd\backslash Z_{\delta}}\omega_{0,z}\left|z\right|^{2}\right] 
  \;<\;\delta\, .
\end{align*}
Then we split the sum in \eqref{equ:TestWithGradient} into a sum over $z\in Z_\delta$ and a sum over $z\notin Z_\delta$.

For $z\in Z_\delta$ we observe that, since $v\in C_{\rm c}^\infty (\Rd)$, we have
\begin{align}
  \frac{v(x+\eps z)-v(x)}{\eps}-\nabla v (x)\cdot z\to0
  \label{equ:uniformTaylor}
\end{align}
uniformly in $x\in\Rd$.
Further, we observe that there exists $\eps^\ast>0$ such that for $z\in Z_\delta$ and for all $\eps<\eps^\ast$, the statement $\eps x\notin 2 \supp v$ implies that $\eps x+\eps z\notin \supp v$.
It follows that for $\eps$ small enough, we have
\begin{align*}
  &\eps^d \sum_{x\in \Z^d}\sum_{z\in Z_{\delta}}\omega_{x,z}\left(\frac{v(\eps x+\eps z)-v(\eps x)}{\eps}-\nabla v(\eps x)\cdot z\right)^{2}\\
  &\qquad\leq \eps^d \sum_{x\in 2\eps^{-1}\supp v}\sum_{z\in Z_{\delta}}\omega_{x,z}\left(\frac{v(\eps x+\eps z)-v(\eps x)}{\eps}-\nabla v (\eps x)\cdot z\right)^{2}\, .
\end{align*}
This together with \eqref{equ:uniformTaylor} and the ergodic theorem implies that
\begin{align*}
  \eps^d \sum_{x\in\Zd}\sum_{z\in Z_{\delta}}\omega_{x,z}\left(\frac{v(\eps x+\eps z)-v(\eps x)}{\eps}-\nabla v(\eps x)\cdot z\right)^{2}
  \to 0\qquad\text{as }\eps\to 0\, .
\end{align*}

Let us now consider the case $z\notin Z_\delta$.
As in the proof of Lemma \ref{lem:BddDiscrGrad}, we interchange the sums 
and observe that
\begin{align*}
  \eps^d \sum_{x\in\Zd}\sum_{z\notin Z_\delta} &\omega_{x,z}\left(\frac{v(\eps x+\eps z)-v(\eps x)}{\eps}-\nabla v(\eps x)\cdot z\right)^{2}\\
  &\qquad\qquad\leq 4\| \nabla v \|_\infty^2  \eps^d \sum_{z\notin Z_\delta} \sum_{x\in\eps^{-1} (\supp v \cup \supp v - \eps z)}\omega_{x,z} |z|^2\\
  &\qquad\qquad\leq 4\| \nabla v \|_\infty^2  \eps^d \sum_{x\in\eps^{-1} \supp v }\sum_{z\notin Z_\delta}\left( \omega_{x,z} |z|^2 + \omega_{x,-z} |z|^2 \right)
\end{align*}
By the ergodic theorem and the choice of $Z_\delta$ it follows that the limit superior of the above RHS is bounded from above by a constant times $\delta |\supp v|$, which we can choose arbitrarily small.
\end{proof}

\begin{cor}[of Lemma \ref{lem:TestWithGradient}]
\label{cor:TestWithGradient}
For all typical realizations $\omega\in\Omega_{\Phi}$ it holds:
If $\rmw_\eps \stackrel{\rm 2s}{\weakto} \rmw$, then
\begin{align}
\lim_{\eps\to 0}\eps^d \sum_{x\in \Zd} \sum_{z\in\Zd}\w_{x,z} \rmw_{\eps}(\eps x,z) \del_z^\eps v(\eps x) = \int_{\Rd} \E\left[\sum_{z\in\Zd}\w_{0,z}\rmw(x,\w,z)\left(\nabla v (x)\cdot z\right) \right]\, \d x
\label{eq:two-sc-conv-2}
\end{align}
for all $v\in  C_{\rm c}^\infty (\Rd)$.
\end{cor}

\begin{proof}
First we observe that $z_{i}\in\Phi$ for $i=1,\dots,d$ and
$\partial_{\vec{e}_i}v\in C_{\rm c}^\infty (\Rd)$ where the $\vec{e}_i$, $i=1,\ldots,d$, denote the unit base vectors of $\Rd$.
Therefore the assumption that $\rmw_\eps \stackrel{\rm 2s}{\weakto} \rmw$ implies that
\begin{align*}
\lim_{\eps\to 0}\eps^d \sum_{x\in \Zd} \del_{\vec{e}_i} v(\eps x)\sum_{z\in\Zd}\w_{x,z} z_i \rmw_{\eps} (\eps x,z) =\int_{\Rd} \del_{\vec{e}_i} v(x) \E\left[\sum_{z\in\Zd} \w_{0,z} z_i \rmw(x,\w,z) \right]\d x\,,
\end{align*}
where the $\vec{e}_i$, $i=1,\ldots,d$, denote the unit base vectors of $\Rd$.
It follows that
\begin{align*}
\lim_{\eps\to 0}\eps^d \sum_{x\in \Zd} \nabla v(\eps x) \cdot \sum_{z\in\Zd}\w_{x,z} z \rmw_{\eps} (\eps x,z) =\int_{\Rd} \nabla v(x)\cdot \E\left[\sum_{z\in\Zd} \w_{0,z} z \rmw(x,\w,z) \right]\d x\,
\end{align*}
for all $v\in C_{\rm c}^\infty (\Rd)$.
In order to prove (\ref{eq:two-sc-conv-2}), it thus remains to show that
\begin{align*}
  \lim_{\eps\to 0}\left|\eps^d \sum_{x\in \Zd}  \sum_{z\in\Zd}\w_{x,z} \left(\del_z^\eps v (\eps x)-\nabla v(\eps x) \cdot z\right) \rmw_{\eps} (\eps x,z)\right| \to 0\, .
\end{align*}
This follows from Cauchy-Schwarz, i.e.,
\begin{align*}
  &\left|\eps^d \sum_{x\in \Zd}  \sum_{z\in\Zd}\w_{x,z} \left(\del_z^\eps v (\eps x)-\nabla v(\eps x) \cdot z\right) \rmw_{\eps} (\eps x,z)\right|\\
  &\qquad\qquad\leq \left(\eps^d \sum_{x\in\Zd} \sum_{z\in\Zd} \w_{x,z} \rmw^2_\eps (\eps x,z)\right)^{1/2} \left(\eps^d \sum_{x\in \Z^d}\sum_{z\in \Zd}\omega_{x,z}\left(\del^\eps_z v(\eps x)-\nabla v(\eps x)\cdot z\right)^{2}\right)^{1/2}\, .
\end{align*}
The first factor on the RHS is bounded by assumption and the second factor converges to zero by virtue of Lemma \ref{lem:TestWithGradient}.
\end{proof}

\subsection{Convergence of gradients}
Let us start with the following auxiliary lemma
\begin{lem}
\label{lem:bsolOrthGrad}
  For all $\omega\in\Omega_\Phi$ and all $b\in\Phi_{\rm sol}$ the following is true:
  \begin{align}
    \sum_{x\in \Zd} \sum_{z\in\Zd}\partial_{z}^{\eps}v(\eps x)\,\omega_{x,z}b(\tau_x\omega, z)=0 \qquad\text{for all }v\in \ell^\infty (\Zdeps)\text{ with bounded support}\,.
    \label{equ:bsolOrthGrad}
  \end{align}
\end{lem}
\begin{proof}
We write the LHS of \eqref{equ:bsolOrthGrad} as
\begin{align*}
  \eps^{-1}\sum_{x\in \Zd} \sum_{z\in \Zd} v(\eps x + \eps z) \w_{x,z} b(\tau_x\w, z) -\eps^{-1}\sum_{x\in \Zd} v(\eps x) \sum_{z\in \Zd}  \w_{x,z} b(\tau_x\w, z) 
\end{align*}
The second term is immediately zero since $\omega\in\Omega_\Phi$ and ${\rm div} (\omega b)\circ\tau_x = 2 \sum_{z\in\Zd} \omega_{x,z} b(\tau_x\w, z)$ by \eqref{eq:divergence}.
The first term is absolutely convergent and thus we can interchange the sums.
After an additional index shift $x\rightsquigarrow x-z$, we obtain that the above first term is equal to
\begin{align*}
  \eps^{-1}\sum_{x\in \Zd} v(\eps x)\sum_{z\in \Zd}  \w_{x-z,z} b(\tau_{x-z}\w, z)
  & = -\eps^{-1}\sum_{x\in \Zd} v(\eps x)\sum_{z\in \Zd}  \w_{x,-z} b(\tau_{x}\w, -z)\,,
\end{align*}
where we have used \eqref{equ:shiftcov2} as well as the symmetry of the conductances, i.e., $\w_{x-z,z}=\w_{x,-z}$.
The claim follows from \eqref{eq:divergence} and $b\in\Phi_{\rm sol}$ since $\omega\in\Omega_\Phi$.
\end{proof}

We can now prove the convergence of gradients. Our result is the natural transfer to the corresponding original result by Nguetseng \cite{nguetseng1989general} to the present setting.
\begin{lem}[Two-scale convergence for gradients]
\label{lem:two-scale-gradient}
For all $\omega\in\Omega_{\Phi}$ such that the Poincar\'e-inequality \eqref{equ:PI} holds uniformly in $\eps$ also the following holds true. If $\ue:\,\Zdeps\to\R$ is a family of functions
with ${\rm supp}(\ue)\subseteq Q\cap \Zdeps$ for all $\eps$ and 
\begin{align}
\sup_{\eps>0}\left( \eps^d \sum_{x\in\Zd} \sum_{z\in\Zd} \w_{x,z} \left( \del_z^\eps u^\eps (\eps x)\right)^2+\left\Vert u^{\eps}\right\Vert _{\infty}\right)
\;<\;\infty\,,
\label{cond:lem:two-scale-gradient}
\end{align}
then there exists a subsequence $u^{\eps'}$, $u\in H_{0}^{1}(Q)$
and $\cgu\in L^{2}(\Rd;L_{\rm pot}^{2})$ such that
\begin{equation}
\mathcal{R}^\ast_{\eps'} u^{\eps'}\weakto u\text{ in }L^2(\Rd)\,,\qquad\del_z^{\eps'}u^{\eps'}(x)\stackrel{2s}{\weakto}\nabla u(x)\cdot z+\cgu(x,\omega,z)\quad\text{as }\eps'\to0\,.
\label{eq:lem:two-scale-gradient}
\end{equation}
Furthermore, if the compact embedding of Lemma \ref{lem:precompact} holds, we find $\mathcal{R}^\ast_{\eps'} u^{\eps'}\to u$ strongly in $L^2(\Rd)$.
\end{lem}
\begin{proof}
Condition \eqref{cond:lem:two-scale-gradient} together with Lemma \ref{lem:precompact} implies that there exists a subsequence, which we still index by $\eps$, and $u\in L^2 (Q)$ such that $\Repsadj u^{\eps}\to u$ in $L^2(\Rd)$.
It remains to show that $u\in H_0^1 (Q)$ and to proof the second statement in \eqref{eq:lem:two-scale-gradient}.

By virtue of Proposition \ref{prop:Ex-two-scale-limit-farfield}, Condition \eqref{cond:lem:two-scale-gradient} further implies that there exists a subsequence, which we still index by $\eps\to0$, and $\rmw\in L^{2}(\Rd;L_{\rm cov}^{2})$ such that $\nabla^{\eps}u^{\eps}\stackrel{2s}{\weakto}\rmw$ in the two-scale sense.
We choose $b\in\Phi_{\rm sol}$ and $v\in C_{\rm c}^\infty (\Rd)$ and apply \eqref{equ:bsolOrthGrad} to the discrete product rule
\begin{align*}
\del_z^\eps (v\ue) (\eps x) &= v(\eps x) \del^\eps_z \ue (\eps x) + \ue (\eps x + \eps z) \del_z^\eps v(\eps x)
\end{align*}
to obtain that
\begin{align}
0 & =\eps^d \sum_{x\in\Zd} \sum_{z\in\Zd} \w_{x,z} \Bigl( v(\eps x) \del^\eps_z \ue (\eps x) + \ue (\eps x + \eps z) \del_z^\eps v(\eps x)\Bigr) b(\tau_x\w, z)\, .
\label{eq:wo-scale-gradient-help-1}
\end{align}

For the first term on the RHS of \eqref{eq:wo-scale-gradient-help-1}, we obtain from the two-scale convergence of $\nabla^{\eps}u^{\eps}$ that 
\begin{align}
\eps^d \sum_{x\in\Zd} v(\eps x) \sum_{z\in\Z^d}\w_{x,z} \del_z^\eps \ue (\eps x) b(\tau_x \w, z)
\to\int_{\Rd} v(x) \E\left[ \sum_{z\in\Zd} \w_{0,z} \rmw(x,\omega,z) b(\w,z) \right]\d x\,.
\label{eq:wo-scale-gradient-help-2}
\end{align}

For the second term on the RHS of \eqref{eq:wo-scale-gradient-help-1}, we first notice that the sum is absolutely convergent since
\begin{align}
   &\eps^d \sum_{x\in\Zd} \sum_{z\in\Zd} \left| \w_{x,z} b(\tau_{x}\w,z) \ue (\eps x + \eps z) \left( \frac{v(\eps x +\eps z) - v(\eps x)}{\eps} \right)\right|\nonumber\\
   &\qquad\qquad=
   \eps^d \sum_{z\in\Zd} \sum_{x\in \eps^{-1} Q_\eps - z} \left| \w_{x,z} b(\tau_{x}\w,z) \ue (\eps x + \eps z) \left( \frac{v(\eps x +\eps z) - v(\eps x)}{\eps} \right)\right|\nonumber\\
   &\qquad\qquad=
   \eps^d \sum_{z\in\Zd} \sum_{x\in \eps^{-1} Q_\eps} \left| \w_{x-z,z} b(\tau_{x-z}\w,z) \ue (\eps x) \left( \frac{v(\eps x) - v(\eps x -\eps z)}{\eps} \right)\right|\nonumber\\
   &\qquad\qquad=
   \eps^d \sum_{x\in \eps^{-1} Q_\eps} \sum_{z\in\Zd} \left| \w_{x,z} b(\tau_{x}\w,z) \ue (\eps x) \left( \frac{v(\eps x +\eps z) - v(\eps x)}{\eps} \right)\right|\, ,
   \label{equ:absconvgrad}
\end{align}
where for the last equality we have used the relation $\w_{x-z,z} = \w_{x,-z}$, the shift covariance \eqref{equ:shiftcov2} and the substitution $z\rightsquigarrow -z$.
We now use the fact that $\ue$ and $\nabla v$ are bounded, apply the Cauchy-Schwarz inequality and the ergodic theorem to obtain that the above sum is indeed finite.
It follows that for the second term on the RHS of \eqref{eq:wo-scale-gradient-help-1}, we can exchange the order of the sums.
By the same arguments as those that led to \eqref{equ:absconvgrad}, we obtain that
\begin{align*}
  \eps^d \sum_{x\in\Zd} \sum_{z\in\Zd} \w_{x,z} \ue (\eps x + \eps z) \del_z^\eps v(\eps x)b(\tau_x\w, z)
  &=\eps^d \sum_{x\in\Zd} \sum_{z\in\Zd} \w_{x,z} \ue (\eps x) \del_z^\eps v(\eps x)b(\tau_x\w, z)\, .
\end{align*}
Further we notice that since $\ue$ has support only in $\eps^{-1}Q_\eps$, we can estimate
\begin{align*}
  \eps^d &\sum_{x\in\Zd} \ue (\eps x) \sum_{z\in\Zd} \w_{x,z} \left( \del^\eps_z v (\eps x) - \nabla v(\eps x)\cdot z \right) b(\tau_x\w, z)\\
  &\leq\| \ue \|^2_\infty\left( \eps^d\mspace{-6mu} \sum_{x\in \Z^d}\sum_{z\in \Zd}\omega_{x,z}\left(\del^\eps_z v(\eps x)-\nabla v(\eps x)\cdot z\right)^{2}\right)^{1/2} \left( \eps^d\mspace{-6mu} \sum_{x\in \eps^{-1}Q_\eps} \sum_{z\in\Zd}\w_{x,z} b^2 (\tau_x\w,z)\right)^{1/2}\, .
\end{align*}
The limit superior of the second factor vanishes due to Lemma \ref{lem:TestWithGradient} and the third factor is finite due to the ergodic theorem.
Thus, for $\omega\in \Omega_{\Phi}$ we have
\begin{align}
  \limsup_{\eps\to 0} \eps^d \sum_{x\in\Zd} \ue (\eps x) \sum_{z\in\Zd} \w_{x,z} \left( \del^\eps_z v (\eps x) - \nabla v(\eps x)\cdot z \right) b(\tau_x\w, z) = 0\, .
  \label{equ:DamnDiscrChainHelp}
\end{align}

To summarize, for the second term on the RHS of \eqref{eq:wo-scale-gradient-help-1}, it follows that
\begin{align*}
  \lim_{\eps\to 0}\eps^d \sum_{x\in\Zd} \sum_{z\in\Zd} &\w_{x,z} \ue (\eps x + \eps z) \del_z^\eps v(\eps x)b(\tau_x\w, z)\\
  &=\lim_{\eps\to 0}\eps^d \sum_{x\in\Zd} \ue (\eps x) \nabla v(\eps x)\cdot \left( \sum_{z\in \Z^d} z\, \w_{x,z} b(\tau_x \w, z) \right)\, .
\end{align*}
By the assumptions on $\omega$ and $b$, the last bracket on the above RHS is in $L^1 (\Omega, \P)$.
Since we already know that the subsequence $\Repsadj u^{\eps}\to u$ in $L^2 (\Rd)$, there exists a further subsequence, which we still index by $\eps\to 0$, where $u^\eps$ converges pointwise a.e.\ in $Q$ \cite[Theorem 4.9]{Brezis2011}.
Moreover, $\ue$ has support in $Q_\eps$ and $\sup_{\eps>0} \| \ue \|_\infty < \infty$ by assumption.
It follows that we can apply Theorem \ref{thm:General-ergodic-thm} along the above subsequence and obtain that
\begin{align}
  \eps^d \sum_{x\in\Zd} \ue (\eps x) \nabla v(\eps x)\cdot \left( \sum_{z\in \Z^d} z\, \w_{x,z} b(\tau_x \w, z) \right)
  \to \int_{\Rd} u(x) \nabla v(x)\cdot \E\left[ \sum_{z\in\Zd}z\w_{0,z} b(\omega,z) \right] \d x\, .
\end{align}
Thus we obtain by \eqref{eq:wo-scale-gradient-help-1} that
\begin{align}
\int_{\Rd} v(x) \E\left[ \sum_{z\in\Zd} \w_{0,z} \rmw(x,\omega,z) b(\w,z) \right]\d x
=-\int_{\Rd} u(x) \nabla v(x)\cdot \E\left[ \sum_{z\in\Zd}z\w_{0,z} b(\omega,z) \right] \d x\,.
\label{eq:wo-scale-gradient-help-3}
\end{align}

Let us now argue that \eqref{eq:wo-scale-gradient-help-3} implies that $\nabla u \in L^{2}(\Rd)$.
By virtue of \eqref{eq:Non-degeneracy-consequence}, for any $i=1,\ldots, d$ we can choose $b^i$ such that $\E\left[\sum_{z\in\Zd} \w_{0,z} zb^i (\w,z)\right]= \vec{e}_i$.
Then \eqref{eq:wo-scale-gradient-help-3} implies that
\begin{align*}
  \left| \int_{\Rd} u(x) \del_i v(x) \d x \right|
  &=\left|\int_{\Rd} v(x) \E\left[ \sum_{z\in\Zd} \w_{0,z} \rmw(x,\omega,z) b^i(\w,z) \right]\d x\right|\\
  &\leq \| b^i \|_{L^2_{\rm cov}}\int_{\Rd} \left|v(x)\right| \left(\E\left[\sum_{z\in\Zd}\w_{0,z} \rmw^2(x,\w,z)\right]\right)^{1/2} \d x\\
  &\leq \|v\|_2 \| b^i \|_{L^2_{\rm cov}} \left(\int_{\Rd}  \E\left[\sum_{z\in\Zd}\w_{0,z} \rmw^2(x,\w,z)\right] \d x\right)^{1/2}\, .
\end{align*}
Since $\rmw \in L^2 (\Rd, L^2_{\rm cov})$, there exists $C<\infty$ such that for any $i=1,\ldots, d$ we observe that
\begin{align*}
  \left| \int_{\Rd} u(x) \del_i v(x) \d x \right|
  &\leq C \|v\|_2\, .
\end{align*}
By virtue of \cite[Proposition 9.3]{Brezis2011} it follows that $\nabla u \in L^2 (\Rd)$.
Since $u|_{\Rd\backslash Q}=0$, we conclude $u\in H_{0}^{1}(Q)$.

We now use integration by parts on the RHS of \eqref{eq:wo-scale-gradient-help-3} and obtain that
\begin{align}
\int_{\Rd} v(x) \E\left[ \sum_{z\in\Zd} \w_{0,z} \left(\rmw(x,\omega,z) - \nabla u (x)\cdot z\right) b(\w,z) \right]\d x=0\,.
\end{align}

Since the last equation holds for all $v\in\Psi$ and all $b\in\Phi_{\rm sol}$,
we find that 
\[
\rmw(x,\w,z)=\nabla u(x)\cdot z+\cgu(x,\w,z)\qquad\mbox{with }\cgu\in L^{2}(\Rd;L_{\rm pot}^{2})\,.
\]
\end{proof}

\section{Proof of Theorem \ref{thm:twoscale-strong}}
\label{ref:ProofPoisson}
We start with an auxiliary lemma.
\begin{lem}
\label{lem:twoscale-weak}
Let $f^{\eps}\colon Q\cap\Zdeps\to\R$ be a sequence of functions such that $\Repsadj f^{\eps}\weakto f$ weakly in $L^{2}(Q)$ for some $f\in L^{2}(Q)$ and such that $\sup_{\eps}\left\Vert f^{\eps}\right\Vert_{\infty}<\infty$.
Then for almost all $\omega\in\Omega$ it holds: The sequence
of solutions $\ue\in\Heps$ to the problem \eqref{equ:epsPoisson} satisfies $\Repsadj \ue\to u$ strongly in $L^{2}(Q)$,
where $u\in H_{0}^{1}(Q)\cap H^{2}(Q)$ solves the limit problem \eqref{eq:strong-limit-equation}.
\end{lem}

\begin{proof}
We test Equation \eqref{equ:epsPoisson} with an arbitrary test function $g^\eps\,:\,\Zdeps\to\R$ with $\supp g^\eps \subseteq Q\cap \Zdeps$ and obtain by Notation \eqref{equ:DiscreteDivForm} and Equation \eqref{eq:Laeps-strictly-positive-def} that
\begin{equation}
\left\langle -\Laeps\ue,g^\eps\right\rangle_\Heps=\left\langle A_\omega^\eps \nabla^\eps \ue,\nabla^\eps g^\eps\right\rangle_\Heps=\left\langle f^\eps,g^\eps\right\rangle_\Heps\,.\label{eq:thm-l-infty-proof-1}
\end{equation}
We now choose $g^\eps =\ue$ and apply \eqref{equ:PI} and Cauchy-Schwarz to obtain that 
\begin{align}
\left\Vert \ue\right\Vert _{\Heps}^{2}\leq C\eps^d \sum_{x\in\Zd} \sum_{z\in\Zd}\w_{x,z}\left( \del_z^\eps u^\eps (\eps x)\right)^2
\leq 2C\left\Vert \ue\right\Vert _{\Heps}\left\Vert f^{\eps}\right\Vert _{\Heps}\,.\label{eq:General-discrete-main-apriori-estimate}
\end{align}
Hence, in combination with Remark \ref{rem:linftyPoiss} we conclude that
\begin{subequations}
\label{equ:PfThm15Bddness}
\begin{equation}
\left\Vert \ue\right\Vert _{\ell^{2}(Q\cap\Zdeps)}^{2}+\eps^d \sum_{x\in\Zd} \sum_{z\in\Zd}\w_{x,z}\left( \del_z^\eps u^\eps (\eps x)\right)^2
\leq 4C^{2}\sup_{\eps>0}\left\Vert f^{\eps}\right\Vert _{\ell^{2}(Q\cap\Zdeps)}^{2}\label{equ:PfThm15BddnessL2}
\end{equation}
\begin{equation}
\text{and }\sup_{\eps>0}\left\Vert \ue\right\Vert_{\infty}
< \infty\,.\label{equ:PfThm15BddnessLinf}
\end{equation}
\end{subequations}
It follows that by virtue of Lemma \ref{lem:precompact} and Lemma \ref{lem:two-scale-gradient}, there exists $u\in H_{0}^{1}(Q)$, $\cgu\in L^{2}(Q;L_{\rm pot}^{2})$ and
a subsequence, which we still index by $\eps$, such that
\begin{equation}
\Repsadj u^{\eps}\to u\,,\mbox{ strongly in }L^{2}(Q)\quad\mbox{and}\quad\del_z^{\eps}u^{\eps} (x)\stackrel{2s}{\weakto}\nabla u (x)\cdot z+\cgu(x,\omega,z)\text{ as }\eps\to0
\label{eq:thm-two-scale-lim-help-1}
\end{equation}
for all $x,z\in \Zd$ and $\omega\in \Omega_\Phi$.

Let us choose $v\in C_{\rm c}^\infty (\Rd)$ with $\supp v\in Q$ and $\varphi\in \Phi_{\rm pot}$ with $\varphi=D \tilde{\varphi}$ for some bounded local function $\tilde{\varphi}$.
When we insert $g^\eps=\eps v\tilde{\varphi}$ into \eqref{eq:thm-l-infty-proof-1}, then we observe for all $\eps>0$ that
\begin{align}
\eps^d \sum_{x\in\Zd} 2f^{\eps}&(\eps x)\left(\eps v (\eps x)\tilde \varphi(\tau_x\omega)\right)\nonumber\\
&= \eps^d \sum_{x\in\Zd} \sum_{z\in \Z^d} \omega_{x,z} \partial_{z}^{\eps}\ue(\eps x)\, \left( v(\eps x + \eps z)\tilde{\varphi} (\tau_{x+z}\omega) - v(\eps x)\tilde{\varphi} (\tau_{x}\omega) \right)\nonumber\\
&= \eps^d \sum_{x\in\Zd} \sum_{z\in \Z^d} \omega_{x,z} \partial_{z}^{\eps}\ue(\eps x)\, \left[ v(\eps x)\left( \tilde{\varphi} (\tau_{x+z}\omega) - \tilde{\varphi} (\tau_{x}\omega) \right) + \eps \tilde{\varphi} (\tau_{x+z}\omega) \del_z^\eps v(\eps x)\right]\nonumber\\
&= \eps^d \sum_{x\in\Zd} \sum_{z\in \Z^d} \omega_{x,z} \partial_{z}^{\eps}\ue(\eps x)\, v(\eps x)\varphi (\tau_{x}\omega, z) \nonumber\\
&\qquad\qquad+\eps^d \sum_{x\in\Zd} \sum_{z\in \Z^d} \omega_{x,z} \partial_{z}^{\eps}\ue(\eps x)\, \eps \tilde{\varphi} (\tau_{x+z}\omega) \del_z^\eps v(\eps x)\, .
\label{equ:PfThm15help2}
\end{align}
The second summand on the above RHS vanishes as $\eps\to 0$ since
\begin{align}
  \eps^d \Big|\sum_{x\in\Zd} \sum_{z\in \Z^d} &\omega_{x,z} \partial_{z}^{\eps}\ue(\eps x)\, \eps \tilde{\varphi} (\tau_{x+z}\omega) \del_z^\eps v(\eps x)\Big|
  \leq \eps^{d+1} \| \tilde\varphi \|_\infty \left|\sum_{x\in\Zd} \sum_{z\in \Z^d} \omega_{x,z} \partial_{z}^{\eps}\ue(\eps x)\, \del_z^\eps v(\eps x)\right|\nonumber\\
  &\leq \eps \| \tilde\varphi \|_\infty \left( \eps^d \sum_{x\in\Zd} \sum_{z\in\Zd}\w_{x,z}\left( \del_z^\eps u^\eps (\eps x)\right)^2 \right)^{1/2} \left( \eps^d \sum_{x\in\Zd}\sum_{z\in \Zd}\omega_{x,z}\left(\del_z^\eps v(\eps x)\right)^{2} \right)^{1/2}
  \label{equ:PfThm15help1}
\end{align}
By assumption $\| \tilde\varphi \|_\infty$ is bounded.
The second factor is bounded due to \eqref{equ:PfThm15Bddness} and the third factor is bounded by virtue of Lemma \ref{lem:BddDiscrGrad}.
Since the LHS of \eqref{equ:PfThm15help2} vanishes as well, \eqref{eq:thm-two-scale-lim-help-1} and \eqref{equ:PfThm15help2} imply that in the limit $\eps\to 0$ and along the chosen subsequence we obtain
\begin{align}
\int_{Q} v(x) \E\left[\sum_{z\in\Zd}\w_{0,z}\left(\nabla u(x)\cdot z + \cgu(x,\omega,z)\right)\varphi(\omega,z) \right]
\d x=0\,.\label{eq:thm-two-scale-lim-help-split-2}
\end{align}
Since $\Phi_{\rm pot}$ is dense in $L^2_{\rm pot}$ and $\Psi$ is dense in $H^1_0(Q)$, Equation \eqref{eq:thm-two-scale-lim-help-split-2} holds for all $\varphi\in L^2_{\rm pot}$ and all $v\in H^1_0(Q)$.

Let $\chi\in \left( L_{\rm pot}^{2}\right)^d$ be given through \eqref{eq:definition-corrector}.
Since $u\in H^1_0 (Q)$ is given, the function $\cgu(x,\tilde{\omega},z):= \nabla u(x)\cdot\chi(\tilde{\omega},z)$
is the unique solution to \eqref{eq:thm-two-scale-lim-help-split-2}.
We have thus identified $\cgu$.

Now we observe that if we test \eqref{eq:thm-l-infty-proof-1} by an arbitrary $g\in C_{\rm c}^\infty (\Rd)$ with support in $Q$, we obtain that
\begin{align*}
\eps^d\sum_{x\in\Zd}\sum_{z\in\Zd}\omega_{x,z}\,\partial_{z}^{\eps}\ue(\eps x)\, \partial_{z}^{\eps}g (\eps x)
=\eps^d \sum_{x\in \Zd} 2f^{\eps}(\eps x) g (\eps x)\,.
\end{align*}
Passing to the limit, we obtain by virtue of Corollary \ref{cor:TestWithGradient} and $\cgu(x,\omega,z) = \nabla u(x)\cdot\chi(\omega,z)$ that
\begin{align}
\int_{\Rd}\E\left[ \sum_{z\in\Zd}\w_{0,z}\left( \nabla u(x) \cdot (z +  \chi)\right) \left(\nabla g (x)\cdot z \right)\right]\, \d x
=\int_{\R^d} 2f(x)g (x)\,\d x\,.
\label{eq:thm-two-scale-lim-help-split-1}
\end{align}

When we now insert $v=\del_i g$ and $\varphi=\chi_i$ for $i=1,\ldots,d$ into \eqref{eq:thm-two-scale-lim-help-split-2} and add the resulting equations to \eqref{eq:thm-two-scale-lim-help-split-1}, then we obtain that
\begin{align}
   \int_{\Rd}\E\left[ \sum_{z\in\Zd}\w_{0,z}\left( \nabla u(x) \cdot (z +  \chi)\right) \left(\nabla g (x)\cdot (z+\chi) \right)\right]\, \d x
   =\int_{\R^d} 2f(x)g (x)\,\d x\, .
\end{align}
A comparison with the definition of $A_{\rm hom}$ in \eqref{eq:Definition-A-hom} finally yields that $u$ solves
\begin{equation}
\int_{Q}\nabla u\cdot\left(A_{\mathrm{hom}}\nabla g\right)=\int_{Q}2f\,g\qquad\text{for all } g\in C_{\rm c}^\infty (\Rd)\text{ with  }\supp g\subseteq Q\,.\label{eq:limit-weak-formulation}
\end{equation}
Since $A_{{\rm hom}}$ is non-degenerate, we find that \eqref{eq:limit-weak-formulation}
is the weak formulation of \eqref{eq:strong-limit-equation}.
Hence, from elliptic regularity theory \cite[Chapter 6]{Evans2010}, we obtain that $u\in H^{2}(Q)\cap H_{0}^{1}(Q)$.

Since the solution $u$ of \eqref{eq:strong-limit-equation} is unique, it follows that \eqref{eq:thm-two-scale-lim-help-1} holds for the entire sequence.
\end{proof}

As for the last ingredient for the proof of Theorem \ref{thm:twoscale-strong}, we observe the following:
On the cube $Q$ the operator $-\Laeps$ with zero Dirichlet conditions is strictly positive definite (see e.g.\ \eqref{eq:Laeps-strictly-positive-def}) and thus it follows that on $Q$ its inverse $\Beps\colon\, \Heps \to \Heps$ is well-defined.
Similarly, the inverse $\Bz\colon\, \Hz \to \Hz$ of $-\Lz$ on $Q$ is well-defined.
We have the following lemma.
\begin{lem}\label{lem:CondJKOii2}
  The operators $\Beps, \Bz$ are \Pas positive, compact and self-adjoint. The norms $\| \Beps \|$ are \Pas bounded by a constant independent of $\eps$.
\end{lem}
\begin{proof}
Since $A_{\rm hom}$ is positive definite and symmetric, the properties of $\Bz$ follow from the theory of elliptic partial differential equations, see e.g.\ \cite[Chapter 6]{Evans2010}.
  
The operator $\Beps$ is uniformly bounded in $\eps$ by virtue of \eqref{equ:PfThm15BddnessL2}.
Moreover, $\Beps$ is real and symmetric by construction and therefore self-adjoint.
Finally, its range $\Heps$ is finite-dimensional and thus $\Beps$ is compact.
\end{proof}
\begin{proof}[Proof of Theorem \ref{thm:twoscale-strong}]
Let us first show that
\begin{align}
  \lim_{\eps\to 0}\int_{\Rd}\bigl( \Repsadj \ue\bigr)\, v
  =\int_{\Rd} u v\quad\text{for all }v\in C(\overline{Q})\, ,
\end{align}
where $u\in H^{2}(Q)\cap H_{0}^{1}(Q)$ is the solution to \eqref{eq:strong-limit-equation}.
Indeed, since $\Beps$ is self-adjoint, we observe that
\begin{align*}
  \int_{\Rd}\bigl( \Repsadj \ue\bigr)\, v
  = \int_{\Rd}\bigl( \Repsadj \Beps f^\eps\bigr)\, v
  = \left\langle f^\eps,\, \bigl( \Beps\Reps v\bigr)\right\rangle_\Heps\, .
\end{align*}
Since $\Repsadj \Reps v \weakto v$ in $L^2$ and $\sup_{\eps>0} \| \Reps v \|_\infty <\infty$, Lemma \ref{lem:twoscale-weak} implies that $\Beps\Reps v$ converges strongly in $L^2$ to $\Bz v$. It follows that
\begin{align*}
  \lim_{\eps\to 0}\left\langle f^\eps,\, \bigl( \Beps\Reps v\bigr)\right\rangle_\Heps
  = \int_{\Rd} f\,  \bigl( \Bz v\bigr)
  = \int_{\Rd} \bigl(\Bz f\bigr)\, v
  = \int_{\Rd} uv\, ,
\end{align*}
where we have used that the operator $\Bz$ is self-adjoint, see Lemma \ref{lem:CondJKOii2}.

We further note that $\sup_{\eps>0} \| \Repsadj \ue \|_2 <\infty$ by the same arguments as for \eqref{equ:PfThm15BddnessL2}.
Since $C(\overline{Q})$ is dense in $L^2 (Q)$, it thus follows that
$\Repsadj \ue \weakto u$.
By virtue of Lemma \ref{lem:precompact} and \eqref{equ:PfThm15BddnessL2} we conclude that $\Repsadj \ue \to u$ strongly in $L^2$.
\end{proof}

\section{Proofs of Proposition \ref{prop:twoscale-strong} and Theorem \ref{thm:spectrum}}
\label{sec:ProofSpectrum}
\begin{proof}[Proof of Proposition \ref{prop:twoscale-strong}]
The existence of solutions to \eqref{eq:prop:twoscale-strong-eps} follows from positivity of the first eigenvalue for small $\eps$.
Hence we can calculate the apriori estimates similar to \eqref{eq:General-discrete-main-apriori-estimate} by testing \eqref{eq:prop:twoscale-strong-eps} with $\ue$ and using $\liminf_{\eps\to0}\lambda_1^\eps>0$ to obtain 
\begin{align*}
 \left\Vert \ue\right\Vert _{\Heps}^{2}\leq (\lambda_1^\eps)^{-1}\langle -\Laeps u^\eps+\Reps Vu^\eps, u^\eps\rangle_\Heps 
\leq 2(\lambda_1^\eps)^{-1}\left\Vert \ue\right\Vert _{\Heps}\left\Vert f^{\eps}\right\Vert _{\Heps}\,.
\end{align*}
Since $V$ is bounded, this implies that $\langle -\Laeps u^\eps, u^\eps\rangle_\Heps$ is bounded in $\eps$.
From Lemma \ref{lem:precompact} it follows that $\Reps^\ast \ue \to u$ strongly in $L^2(Q)$ and hence $\Reps^\ast(\Reps V\, \ue)\weakto Vu$.
Hence from Theorem \ref{thm:twoscale-strong} we obtain that $u$ solves \eqref{eq:prop:strong-limit-equation}.
\end{proof}

\begin{proof}[Proof of Theorem \ref{thm:spectrum}]
First, we notice that without loss of generality, we can assume that the function $V$ is nonnegative.
Otherwise, we simply substitute $V$ for $V-\min_{x\in Q} V(x)$ and prove the result for the new $V$.
Then we notice that the substitution has simply shifted the spectrum by the constant $\min_{x\in Q} V(x)$ and the new eigenvectors are the same as the old ones.
Thus, it suffices to prove the claim for $V\geq 0$.
Note that \eqref{eq:Laeps-strictly-positive-def} directly implies that if $V\geq 0$, then $\lambda_1^{\eps}$ is positive.

Then Lemmas \ref{lem:CondJKO} and \ref{lem:CondJKOii2} ensure that Conditions I-IV of \cite[Section 11.1]{JKO1994} are satisfied and Theorem \ref{thm:spectrum} follows by virtue of \cite[Theorems 11.4, 11.5]{JKO1994}.
\end{proof}
As in the paragraph before Lemma \ref{lem:CondJKOii2}, we now define the operators $\Beps (V)$ and $\Bz (V)$ as the inverses of $-\Laeps + \Reps V$ and $-\Lz+V$, respectively.
For $V\geq 0$, we further consider the spectrum of the operators $\Beps (V)$, where we drop the argument ``$(V)$'' for readability:
\begin{align}
\begin{aligned}
	&\psi_k^\eps \in \Heps, \quad \Beps \psi_k^\eps = \mu_k^\eps \psi_k^\eps, \quad k = 1,2,\ldots\, ,\pb{1.5em}\\
	&\mu_1^\eps \geq \mu_2^\eps \geq \ldots \geq \mu_k^\eps \ldots \, , \quad \mu_k^\eps >0 \, ,\pb{1.5em}\\
	&\langle \psi_k^\eps, \psi_l^\eps \rangle_{\Heps} = \delta_{kl}\, ,\pb{1.5em}
\end{aligned}
\end{align}
as well as the spectrum of the operator $\Bz (V)$, where we also drop the argument ``$(V)$'' for readability:
\begin{align}
\begin{aligned}
	&\psi_k^0 \in \Hz, \quad \Bz \psi_k^0 = \mu_k^0 \psi_k^0, \quad k = 1,2,\ldots\, ,\pb{1.5em}\\
	&\mu_1^0 \geq \mu_2^0 \geq \ldots \geq \mu_k^0 \ldots \, , \quad \mu_k^0 >0 \, ,\pb{1.5em}\\
	&\langle \psi_k^0, \psi_l^0 \rangle_{\Heps} = \delta_{kl}\, .\pb{1.5em}
\end{aligned}
\end{align}
\begin{remark}\label{rem:RelationBL}
The eigenfunctions $\left\{ \psi_k^\eps\right\}_k$ of the operator $\Beps$ and the eigenfunctions $\left\{ \psi_k^0\right\}_k$ of the operator $\Bz$
coincide with the eigenfunctions of the operators $-\Laeps + \Reps V$ and $-\Lz+V$, respectively.
Their eigenvalues relate to those of $-\Laeps+\Reps V$ and $-\Lz+V$ by
\begin{align*}
	\mu_k^\eps = \left( \lambda_k^\eps\right)^{-1}\, , \quad \mu_k^0 = \left( \lambda_k^0\right)^{-1} \, \quad k = 1, 2, \ldots\, .
\end{align*}
\end{remark}
\begin{lem}
\label{lem:CondJKO}~
\begin{enumerate}[label={(\roman*)},ref={\thethm~(\roman*)}]
	\item\label{lem:CondJKOi} For any $u\in\Hz$, the following is true:
	\begin{align}
		\| \Reps u \|_{\Heps} \leq \| u \|_{\Hz}\, .
		\label{equ:CompNorms}
	\end{align}
	Further,
	\begin{align}
		\lim_{\eps\to 0}\, \langle u^\eps, v^\eps \rangle_{\Heps} = \langle u, v \rangle_{\Hz}\, .
		\label{equ:ConvSP}
	\end{align}
	provided that $u, v\in\Hz$ and $u^\eps, v^\eps \in\Heps$ and
	\begin{align}
		\lim_{\eps\to 0}\, \| u^\eps - \Reps u \|_{\Heps} = 0, \quad\text{and}\quad \lim_{\eps\to 0}\, \| v^\eps - \Reps v \|_{\Heps} = 0\, .
		\label{equ:Convueps}
	\end{align}
\end{enumerate}
Let $V\colon \Rd \to \R$ be a non-negative, continuous potential.
If Assumptions \ref{ass:Environ} and \ref{ass:Integ}\ref{ass:Integ:q} are fulfilled, then furthermore the following statements hold.
\begin{enumerate}[label={(\roman*)},ref={\thethm~(\roman*)}]
\setcounter{enumi}{1}
	\item\label{lem:CondJKOiii} Let $f\in\Hz$ and let $f^\eps \in\Heps$. Then the following is true: If
	\begin{align}
		\lim_{\eps\to 0}\, \| f^\eps - \Reps f \|_{\Heps} = 0\, ,
	\end{align}
	then
	\begin{align}
	  \lim_{\eps\to 0}\, \| \Beps f^\eps - \Reps \Bz f \|_{\Heps} = 0\quad \mathbb{P}\text{-a.s.}
			\label{equ:ConvSol}			
	\end{align}
	\item\label{lem:CondJKOiv} For any sequence $f^\eps\in\Heps$ such that $\sup_\eps \| f^\eps \|_\Heps<\infty$, there exists a subsequence $f^{\eps'}$ and a vector $w^0\in\Hz$ such that
	\begin{align*}
	  \lim_{\eps'\to 0} \left\| \mathcal{R}^\ast_{\eps'}\mathcal{B}_{\eps'} f^{\eps'} - w^0 \right\|_{\Hz}
	  &=\lim_{\eps'\to 0} \left\| \mathcal{B}_{\eps'} f^{\eps'} - \mathcal{R}_{\eps'} w^0 \right\|_{\mathcal{H}_{\eps'}} = 0\, .
	\end{align*}
\end{enumerate}
\end{lem}

\begin{proof}
For (i): Let $u\in\Hz$. By Jensen's inequality it follows that
\begin{align*}
	\| \Reps u \|^2_{\Heps} 	&= \eps^d \sum_{z\in\Zdeps} \eps^{-2d} \,\,\Biggl(\,\, \int_{\bzeps} u \d x\Biggr)^2
	\leq \eps^d \sum_{z\in\Zdeps} \eps^{-d} \,\,\Biggl(\,\, \int_{\bzeps} u^2 \d x\Biggr)
	= \| u \|^2_{\Hz}\, .
\end{align*}
For \eqref{equ:ConvSP} we first observe that
\begin{align}
	\left| \langle u^\eps, v^\eps \rangle_{\Heps} - \langle u, v \rangle_{\Hz} \right|
		&\leq \left|\left\langle v^\eps, u^\eps - \Reps u \right\rangle_{\Heps}\right| + \Biggr|\sum_{z\in\Zdeps} \int_{\bzeps} u\left(\Reps^\ast v^\eps(z)-v\right) \d x \Biggl|
		\nonumber\\
	&\leq \| v^\eps \|_{\Heps} \| u^\eps - \Reps u \|_{\Heps}
	+ \| u \|_{\Hz}\, \| v^\eps - \Reps v \|_{\Heps}\, .
	\label{equ:EstDiffSP1}
\end{align}
The second term on the above RHS converges to zero by assumption.
For the first term we note that the triangle inequality together with \eqref{equ:CompNorms} yields
\begin{align*}
		\| v^\eps \|_{\Heps} \leq \| \Reps v \|_{\Heps} + \| v^\eps - \Reps v \|_{\Heps} \leq \| v \|_{\Hz} + \| v^\eps - \Reps v \|_{\Heps},
\end{align*}
which is bounded from above.
It follows that the first term on the RHS of \eqref{equ:EstDiffSP1} converges to zero as well.

Part (ii) follows directly from Proposition \ref{prop:twoscale-strong} and \eqref{equ:ConvSP}.

Similarly, Part (iii) follows from Proposition \ref{prop:twoscale-strong} and \eqref{equ:ConvSP} since $\sup_\eps \| f^\eps \|_2 <\infty$ implies that there exists a subsequence $\eps'$ along which $\mathcal{R}^\ast_{\eps'} f^{\eps'} \weakto f$ in $L^2$.
\end{proof}

\section{Proof of Proposition \ref{prop:LDP}}
\label{sec:ProofPropLDP}
\begin{proof}[Proof of Proposition \ref{prop:LDP}]
This proof is an application of the G\"artner-Ellis theorem and goes along the lines of \cite[Theorem 1.8]{KoenigWolff}.
For the convenience of the reader, we outline the main steps here.

Let $V\colon \Rd \to \R$ be a bounded, continuous function.
We define the generating cumulant function
\begin{align}
  \Lambda_t (V) \ldef \frac{\alpha_t^2}{t} \log \Erw^\w_0 \left[ \left. \exp \left\{ -\frac{t}{\alpha_t^2} \int_Q V(y) L_t(y) \,\d y \right\} \,\right|\, X_{[0,t]} \subset \alpha_t Q\right]\, .
\end{align}
As in \cite{KoenigWolff}, it suffices to show that
\begin{align}
  \Lambda (V) \ldef \lim_{t\to\infty} \Lambda_t (V) = -\lambda_1 (V) + \lambda_1 (0)\, ,
  \label{equ:ConvGenCum}
\end{align}
where $\lambda_1 (V)$ denotes the principal Dirichlet eigenvalue of $-\Lz + V$ on $Q$ with zero Dirichlet boundary conditions.
Then the claim follows by the G\"artner-Ellis theorem.

In order to show \eqref{equ:ConvGenCum}, we define the operator $\Pavt$ acting on real-valued functions $f\in\ell^2 (\alpha_t Q\cap \Zd)$ by
\begin{align}
  \left(\Pavt f\right) (z) \ldef \Erw^\w_z \left[ \exp \left\{ -\frac{t}{\alpha_t^2} \int_Q V(y) L_t(y) \,\d y \right\} \mathds{1}_{\{ X_{[0,t]} \subset \alpha_t Q \}} f\left( X_t \right)\right] \quad (z\in\alpha_t Q\cap \Zd)\, .
\end{align}
Since $L_t$ is a step function, $\Pavt$ admits the semigroup representation
\begin{align}
  \Pavt = \exp \left\{ -t\alpha_t^{-2} \left[ -\alpha_t^2 \La + V_t \right] \right\}\, ,
\end{align}
where the operator in the exponent is considered with zero Dirichlet conditions at the boundary of $\alpha_t Q\cap \Zd$ and
\begin{align*}
  V_t (z) \ldef \int_{[-\frac{1}{2}, \frac{1}{2}]} V\left( \frac{z+y}{\alpha_t} \right) \, \d y \qquad (z\in \alpha_t Q \cap \Zd)\, .
\end{align*}

Let $\lambda_1^{(t)} (V)$ denote the principal Dirichlet eigenvalue of $-\alpha_t^2 \La + V_t$ on $\alpha_t Q\cap \Zd$ with zero Dirichlet boundary conditions.
Let $\psi_1^{(t)} (V)$ be the corresponding principal Dirichlet eigenfunction.
Then, in order to show \eqref{equ:ConvGenCum}, we have to show that
\begin{align}
  \lim_{t\to\infty} \frac{\alpha_t^2}{t} \log \left( \Pavt \mathds{1} \right) (0) = \lim_{t\to\infty} \lambda_1^{(t)} (V) = \lambda_1 (V)
\end{align}
for any $V\in\mathcal{C}_{\rm b} (\Rd)$.
The second equality follows by virtue of Theorem \ref{thm:spectrum}.
It remains to prove the first equality.
For this purpose we notice that an eigenvalue expansion together with Cauchy-Schwarz and Parseval's identity yields that
\begin{align*}
  \left( \Pavt \mathds{1} \right) (0) \leq \sqrt{|\alpha_t Q|} \exp\left\{ -\frac{t}{\alpha_t^2} \lambda_1^{(t)} (V) \right\}\, .
\end{align*}
On the other hand, since $\Pavt\geq 0$, we can estimate from below
\begin{align*}
  \left( \Pavt \mathds{1} \right) (0) &\;\geq\; \frac{1}{\sup_{\alpha_t Q} \psi_1^{(t)}} \left( \Pavt \psi_1^{(t)} \right) (0)
    \;\geq\; \psi_1^{(t)} (0) \exp\left\{ -\frac{t}{\alpha_t^2}\, \lambda_1^{(t)} (V)\right\}
\end{align*}
since $\psi_1^{(t)}$ is a normalized eigenfunction.
Thus, if $\psi_1^{(t)} (0)$ decays at most polynomially, we have proved the claim.

Similarly to the proof in \cite{KoenigWolff}, we obtain that
\begin{align}
  \psi_1^{(t)} (0) \geq {\rm e}^{-\lambda (V) - V^\ast} \left( \max_{x\in\alpha_t Q\cap \Zd} \psi_1^{(t)} \right)  \left( \min_{x\in\alpha_t Q\cap \Zd} \Prw^{\alpha_t^2\w}_0 \left[ X_1 =x \right]\right)\, ,
\end{align}
where $V^\ast$ is an upper bound for $V$.
Since $\psi_1^{(t)}$ is normalized and
\begin{align*}
  \min_{x\in\alpha_t Q\cap \Zd} \Prw^{\alpha_t^2\w}_0 \left[ X_1 =x \right] = \min_{x\in\alpha^2_t Q\cap \Zd} \Prw^{\w}_0 \left[ X_{\alpha_t} =x \right]
\end{align*}
decays at most polynomially by Assumption \ref{ass:heatkernel}, the claim follows.
\end{proof}

\section*{Acknowledgements}
The authors thank Wolfgang K\"onig, Stefan Neukamm, Mathias Sch\"affner, and Mark Peletier for fruitful discussions and valuable hints.
Furthermore we wish to thank an anonymous referee for his or her valuable comments and for pointing out several typos in an earlier version of this manuscript.
F.\ F.\ gratefully acknowledges the support by DFG within the RTG 1845 \emph{Stochastic Analysis with Applications in Biology, Finance and Physics}.
M.\ H.\ was supported by DFG through SFB 1114 \emph{Scaling Cascades in Complex Systems}, subproject C05 \emph{Effective models for interfaces with many scales}.

\bibliographystyle{alpha}

\begin{thebibliography}{Kum14}

\bibitem[ADS15]{ADS15}
S.~Andres, J.-D. Deuschel, and M.~Slowik.
\newblock Invariance principle for the random conductance model in a degenerate
  ergodic environment.
\newblock \href{https://doi.org/10.1214/14-AOP921}{{\em Ann. Probab.}, 43(4):1866--1891, 2015}.

\bibitem[ADS16]{ADS16}
S.~Andres, J.-D. Deuschel, and M.~Slowik.
\newblock Harnack inequalities on weighted graphs and some applications to the
  random conductance model.
\newblock \href{https://doi.org/10.1007/s00440-015-0623-y}{{\em Probab. Theory Relat. Fields}, 164:931, 2016}.

\bibitem[All92]{Allaire1992}
G.~Allaire.
\newblock Homogenization and two-scale convergence.
\newblock \href{https://doi.org/10.1137/0523084}{{\em SIAM J. Math. Anal.}, 23(6):1482--1518, 1992}.

\bibitem[BD03]{BoivinDepauw2003}
D.~Boivin and J.~Depauw.
\newblock Spectral homogenization of reversible random walks on
  {$\mathbb{Z}^d$} in a random environment.
\newblock \href{https://doi.org/10.1016/S0304-4149(02)00233-8}{{\em Stochastic {P}rocesses and their {A}pplications}, 104(1):29--56, 2003}.

\bibitem[BFK16]{BFK16}
M.~Biskup, R.~Fukushima and W.~K\"onig.
\newblock Eigenvalue fluctuations for lattice Anderson Hamiltonians.
\newblock \href{https://doi.org/10.1137/14097389X}{SIAM J. Math. Anal., 48(4):2674--2700, 2016}.

\bibitem[BFK17]{BFK17}
M.~Biskup, R.~Fukushima and W.~K\"onig.
\newblock Eigenvalue fluctuations for lattice Anderson Hamiltonians: Unbounded potentials.
\newblock \href{https://arxiv.org/abs/1710.06592}{{\em Preprint, available at arXiv:1710.06592}, 2017}.

\bibitem[BG90]{Bouchaud1990}
J.-P. Bouchaud and A.~Georges.
\newblock Anomalous diffusion in disordered media: statistical mechanisms,
  models and physical applications.
\newblock \href{https://doi.org/10.1016/0370-1573(90)90099-N}{{\em Physics Reports}, 195(4--5):127--293, 1990}.

\bibitem[Bis11]{Biskup2011review}
M.~Biskup.
\newblock Recent progress on the random conductance model.
\newblock \href{https://doi.org/10.1214/11-PS190}{{\em Probability Surveys}, 8:294--373, 2011}.

\bibitem[Bre11]{Brezis2011}
H.~Brezis.
\newblock {\em Functional analysis, {S}obolev spaces and partial differential
  equations}.
\newblock Universitext. Springer, New York, 2011.

\bibitem[DNS18]{DNS18}
J.-D. Deuschel, T.~A. Nguyen, and M.~Slowik.
\newblock Quenched invariance principles for the random conductance model on a
  random graph with degenerate ergodic weights.
\newblock \href{https://doi.org/10.1007/s00440-017-0759-z}{{\em Probab. Theory Relat. Fields}, 170(1--2):363--386, 2018}.

\bibitem[DV75]{DV75}
M.~D. Donsker and S.~R.~S. Varadhan.
\newblock Asymptotic evaluation of certain {M}arkov process expectations for
  large time, {I}.
\newblock \href{https://doi.org/10.1002/cpa.3160280102}{{\em Comm. Pure Appl. Math.}, 28:1--47, 1975}.

\bibitem[Eva10]{Evans2010}
L.~C. Evans.
\newblock {\em Partial differential equations}, volume~19 of {\em Graduate
  Studies in Mathematics}.
\newblock American Mathematical Society, Providence, RI, second edition, 2010.

\bibitem[Fag08]{Faggionato2008}
A.~Faggionato.
\newblock Random walks and exclusion processes among random conductances on
  random infinite clusters: homogenization and hydrodynamic limit.
\newblock \href{https://doi.org/10.1214/EJP.v13-591}{{\em Electron. J. Probab.}, 13:2217--2247, 2008}.

\bibitem[Fag12]{Faggionato2012}
A.~Faggionato.
\newblock Spectral analysis of 1{D} nearest-neighbor random walks and
  applications to subdiffusive trap and barrier models.
\newblock \href{https://doi.org/10.1214/EJP.v17-1831}{{\em Electron. J. Probab.}, 17:no.\ 15, 1--36, 2012}.

\bibitem[Fle16]{Flegel2016}
F.~Flegel.
\newblock Localization of the principal {D}irichlet eigenvector in the
  heavy-tailed random conductance model.
\newblock {\em Preprint, available at \href{https://arxiv.org/abs/1608.02415}{arXiv:1608.02415}}, 2016.

\bibitem[GKS07]{GKS07}
N.~Gantert, W.~K{\"o}nig, and Z.~Shi.
\newblock Annealed deviations of random walk in random scenery.
\newblock \href{https://doi.org/10.1016/j.anihpb.2005.12.002}{{\em Ann. Inst. H. Poincar\'e Probab. Statist.}, 43(1):47--76, 2007}.

\bibitem[GM18]{GM18}
A.~Giunti and J.-C. Mourrat.
\newblock Quantitative homogenization of degenerate random environments.
\newblock \href{http://doi.org/10.1214/16-AIHP793}{{\em Ann. Inst. H. Poincaré Probab. Statist.}, 54(1):22--50, 2018}.

\bibitem[JKO94]{JKO1994}
V.~V. Jikov, S.~M. Kozlov, and O.~A. Ole{\u\i}nik.
\newblock {\em Homogenization of differential operators and integral
  functionals}.
\newblock Springer-Verlag, Berlin, 1994.
\newblock Translated by the Russian by G. A. Yosifian [G. A. Iosif$'$yan].

\bibitem[Kes86]{Kesten1984}
H.~Kesten.
\newblock Aspects of first passage percolation.
\newblock In {\em \'{E}cole d'\'et\'e de probabilit\'es de {S}aint-{F}lour,
  {XIV}---1984}, volume 1180 of {\em Lecture Notes in Math.}, pages 125--264.
  Springer, Berlin, 1986.
  
  \bibitem[Kum14]{Ku10}
T.~Kumagai.
\newblock {\em Random walks on disordered media and their scaling limits},
  volume 2101 of {\em Lecture Notes in Mathematics}.
\newblock \href{https://doi.org/10.1007/978-3-319-03152-1}{Springer, Cham, 2014}.
\newblock Lecture notes from the 40th Probability Summer School held in
  Saint-Flour, 2010, \'Ecole d'\'et\'e de Probabilit\'es de Saint-Flour.
  [Saint-Flour Probability Summer School].

\bibitem[KW15]{KoenigWolff}
W.~K{\"o}nig and T.~Wolff.
\newblock Large deviations for the local times of a random walk among random
  conductances in a growing box.
\newblock \href{http://math-mprf.org/journal/articles/id1380/}{{\em Markov Process. Related Fields}, 21(3, part 1):591--638, 2015}.

\bibitem[Ma02]{Ma2002BanachMeasures}
T.-W. Ma.
\newblock {\em Banach-{H}ilbert spaces, vector measures and group
  representations}.
\newblock World Scientific Publishing Co., Inc., River Edge, NJ, 2002.

\bibitem[Mas93]{DalMaso1993}
G.~Dal~Maso.
\newblock {\em An introduction to {$\Gamma$}-convergence}.
\newblock Progress in Nonlinear Differential Equations and their Applications,
  8. Birkh\"auser Boston, Inc., Boston, MA, 1993.

\bibitem[MP07]{MP07}
P.~Mathieu and A.~Piatnitski.
\newblock Quenched invariance principles for random walks on percolation clusters.
\newblock \href{https://doi.org/10.1098/rspa.2007.1876}{{\em Proc. R. Soc. Lond. Ser. A Math. Phys. Eng. Sci.}, 463(2085):2287--2307, 2007}.

\bibitem[NSS17]{Neukamm2017}
S.\ Neukamm, M.\ Sch{\"a}ffner, A.\ Schl{\"o}merkemper.
\newblock Stochastic homogenization of nonconvex discrete energies with
  degenerate growth.
\newblock \href{https://doi.org/10.1137/16M1097705}{{\em SIAM J. Math. Anal.}, 49(3):1761--1809, 2017}.

\bibitem[Ngu]{nguetseng1989general}
G.~Nguetseng.
\newblock A general convergence result for a functional related to the theory of homogenization.
\newblock \href{https://doi.org/10.1137/0520043}{{\em SIAM J. Math. Anal.}, 20(3):608--623, 1989}.

\bibitem[PZ17]{piatnitski2017periodic}
A.~Piatnitski and E.~Zhizhina.
\newblock Periodic homogenization of nonlocal operators with a convolution-type kernel.
\newblock \href{https://doi.org/10.1137/16M1072292}{{\em SIAM J. Math. Anal.}, 49(1):64--81, 2017}.

\bibitem[Sen06]{SenetaNN}
E.~Seneta.
\newblock {\em Non-negative matrices and {M}arkov chains}.
\newblock Springer Series in Statistics. Springer, New York, 2006.
\newblock Revised reprint of the second (1981) edition [Springer-Verlag, New
  York; MR0719544].

\bibitem[ZP06]{zhikov2006homogenization}
V.~V. Zhikov and A.~L. Pyatnitski{\u\i}.
\newblock Homogenization of random singular structures and random measures.
\newblock \href{https://doi.org/10.1070/IM2006v070n01ABEH002302}{{\em Izvestiya: Mathematics}, 70(1):19--67, 2006}.

\end{thebibliography}

\end{document}